\documentclass[a4paper,11pt,leqno]{article}
\usepackage{etex}
\usepackage[a4paper,left=25mm,right=25mm,top=30mm,bottom=30mm,marginpar=25mm]{geometry}
\usepackage[utf8x]{inputenc}
\usepackage[english]{babel}
\usepackage{amsmath}
\usepackage{amssymb}
\usepackage{amsthm}
\usepackage{enumerate}
\usepackage{empheq}
\usepackage{fancybox}
\usepackage{xcolor}
\usepackage{mathrsfs}
\usepackage{slashed}
\usepackage{tikz}
\usepackage{harpoon}
\usepackage{setspace}
\usepackage{fancybox}
\usepackage{esint}
\usepackage{bm}
\usepackage{hyperref}
\usepackage{subfigure}
\usepackage[all]{xy}
\usepackage{enumitem}

\usepackage{graphicx}

\DeclareMathOperator{\id}{id}

\newcommand{\res}{\mathop{\hbox{\vrule height 7pt width .5pt depth 0pt
\vrule height .5pt width 6pt depth 0pt}}\nolimits}
\renewcommand{\div}{{\rm div}\,}

\newcommand{\R}{\mathbb{R}}

\newcommand{\N}{\mathbb{N}}
\newcommand{\M}{\mathcal{M}}
\newcommand{\A}{\mathcal{A}}

\newcommand{\Na}{\mathcal{N}}

\newtheorem*{thmnonumber}{Theorem}
\newtheorem{thm}{Theorem}[section]
\newtheorem{pro}[thm]{Problem}
\newtheorem{defi}[thm]{Definition}

\newtheorem{prop}[thm]{Proposition}

\newtheorem{lemma}[thm]{Lemma}
\newtheorem{rmk}[thm]{Remark}

\begin{document}
\title{Sparsity of solutions for variational inverse problems with
  finite-dimensional data} \author{Kristian Bredies\footnote{Kristian Bredies, Institute of Mathematics and Scientific Computing,
    University of Graz, Heinrichstra\ss{}e 36, A-8010 Graz,
    Austria. Email: \texttt{kristian.bredies@uni-graz.at}} \and Marcello
  Carioni\footnote{Marcello Carioni, Institute of Mathematics and Scientific Computing,
    University of Graz, Heinrichstra\ss{}e 36, A-8010 Graz,
    Austria. Email: \texttt{marcello.carioni@uni-graz.at}\newline \indent The Institute of Mathematics and Scientific Computing is a member of NAWI Graz (\url{www.nawigraz.at}).}}
  
\maketitle
\begin{abstract}
\sloppy
In this paper we characterize sparse solutions for variational problems of the form $\min_{u\in X} \phi(u) + F(\A u)$, where $X$ is a locally convex space, $\A$ is a linear continuous operator that maps into a finite dimensional Hilbert space and $\phi$ is a seminorm.  More precisely, we prove that there exists a minimizer that is ``sparse'' in the sense that it is represented as a linear combination of the extremal points of the unit ball associated with the regularizer $\phi$ (possibly translated by an element in the null space of $\phi$). We apply this result to relevant regularizers such as the total variation seminorm and the Radon norm of a scalar linear differential operator. In the first example, we provide a theoretical justification of the so-called staircase effect and in the second one, we recover the result in \cite{unsersplines} under weaker hypotheses.
\end{abstract}

\section{Introduction}

One of the fundamental tasks of inverse problems is to reconstruct data from a small number of usually noisy observations. This is of capital importance in a huge variety of fields in science and engineering, where typically one has access only to a fixed and small number of measurements of the sought unknown.
However, in general, this type of problem is underdetermined and therefore, the recovery of the true data is practically impossible. One common way to obtain a well-posed problem is to make \emph{a priori} assumptions on the unknown and, more precisely, to require that the latter is \emph{sparse} in a certain sense.
In this case, the initial data can often be recovered by solving a minimization problem with a suitable regularizer of the form
\begin{equation}\label{inizio}
\inf_{u \in X} \phi(u) \qquad \mbox{ subjected to }\  \A u = y\,,
\end{equation}
where $\phi$ is the regularizer, $\A: X \rightarrow H$, $H$ finite-dimensional Hilbert space, models the finite number of observations (that is small compared to the dimension of $X$) and $y \in H$ is noise-free data.

When the domain $X$ is finite-dimensional and the regularizer is the $\ell_1$ norm, the problem falls into the established theory of \emph{compressed sensing} \cite{compressedsensing, donohocompressed} that has seen a huge development in recent years. In this case, sparsity is intended as a \emph{high number of zero coefficients} with respect to a certain basis of $X$.

In an infinite dimensional setting, when the domain $X$ is usually a Banach space, there has been a clear evidence that the action of the regularizers is promoting different notions of sparsity, but there have not been a comprehensive theory explaining this effect.

Nevertheless, the effect of \emph{sparsity} plays a crucial role in the field of image processing and computer vision: in many cases, the recovered image in a variational model can be interpreted as \emph{sparse}  with respect to %
a notion of \emph{sparsity} that is depending on the regularizer. For example, for classical total variation (TV) denoising \cite{tvdenoising}
\begin{equation}\label{tvintr}
\inf_{u\in BV(\Omega)} TV(u) + \frac{\lambda}2\|u - g\|_{L^2(\Omega)}^2 \,,
\end{equation} 
it has been observed that minimizers are characterized by the so called staircase effect (see for example \cite{staircase, esedoglusparse, Nikolova}) which corresponds to the gradient of the considered image having small support.
Another classical example of sparsity-promoting regularizers is $\ell^1$-penalization. In \cite{sparsetik} the authors study the $\ell^p$ regularizer with $1 \leq p \leq 2$ in Hilbert spaces and they note that the case $p=1$ promotes sparsity with respect to a given basis of the Hilbert space which means that only a finite number of coefficients in the respective basis representation is non-zero. In \cite{resmerita}, $\ell^1$-regularization is used %
in the framework of the \emph{least error method} to recover a sparse solution with a fixed bound on the number of non-zero coefficients.
Finally, it has been noted that suitable $\ell^1$-type regularizers enforce sparsity when data are represented in a \emph{wavelet basis} (see for example \cite{antoniadis, donoho}).

The intrinsic sparsity of infinite-dimensional variational models with finite-dimensional data has been investigated by various authors in specific cases and in different contexts. One of the most important instances can be found in \cite{chandra}: here, the authors 
notice that the regularizer is linked to the convex hull of the set of sparse vectors that we aim to recover. 
This was also noticed in \emph{optimal control} theory (see, for example, \cite{casaskunisch}) and used in practice for developing efficient algorithms to solve optimization problems that are based on the sparsity of the minimizers \cite{bredieslorenz, pikka, walter}.

\medskip
More recently, several authors have investigated deeply the connection between regularizers and \emph{sparsity}. In 2016, Unser, Fageot and Ward in \cite{unsersplines} have studied the case where $\phi(u) = \|Lu\|_{\mathcal{M}}$, $L$ is a \emph{scalar} linear differential  operator and $\|\cdot\|_\M$ denotes the Radon norm. They showed the existence of a \emph{sparse} solution, namely a linear combination of counterimages of Dirac deltas which can be expressed using a fundamental solution of $L$. Also, the work of Flinth and Weiss \cite{exactsolutionsflinth} is worth mentioning, where they give an alternative proof of the result in \cite{unsersplines} with less restrictive hypotheses. In both works, however, the case of a vector-valued differential operator was not treated  and therefore, problems involving the total variation regularizer were not covered. 
After this manuscript was finalized, we discovered a recent preprint \cite{chambollerepresenter} where the authors study a similar abstract problem and apply it, in particular, to the TV regularizer in order to justify the staircase effect. We remark that \cite{chambollerepresenter} and the present paper were developed independently and differ in terms of the proofs as well as the applications.

\medskip
In this paper, we provide a theory that characterizes \emph{sparsity} for minimizers of general linear inverse problems with finite-dimensional data constraints. More precisely, we choose to work with locally convex spaces in order to deal, in particular, with weak* topologies. The latter is necessary in order to treat variational problems with TV regularization or Radon-norm regularization.
We consider the following problem:
\begin{equation}\label{prointro}
\inf_{u\in X} \phi(u) + F(\A u)\,,
\end{equation}
where $X$ is a locally convex space, $\phi : X \rightarrow [0,+\infty]$ is a lower semi-continuous seminorm, $\A : X \rightarrow H$ is a linear continuous map with values in a finite-dimensional Hilbert space $H$ and $F$ is a proper, convex, lower semi-continuous functional. (Notice that this generality allows problems of the type \eqref{inizio} for noise-free data as well as soft constraints in case of noisy data.)
Additionally we ask that $\A(\mbox{dom}\, \phi) = H$ (see Assumption \ref{H0} below) and that $\phi$ is coercive when restricted to the quotient space of $X$ with the null-space of $\phi$ that we denote by $\mathcal{N}$ (see Assumption \ref{H1} below).
Under these hypotheses we prove that there exists a \emph{sparse} minimizer of \eqref{prointro}, namely a minimizer that can be written as a linear combination of extremal points of the unit ball associated to $\phi$ (in the quotient space $X/\mathcal{N}$). More precisely, we obtain the following result:

\begin{thmnonumber}[Theorem~\ref{maint}]
Under the previous hypotheses there exists $\overline u \in X$, a minimizer of \eqref{prointro} such that:
\begin{equation}
\overline{u} = \overline \psi + \sum_{i=1}^p \gamma_i u_i \, ,
\end{equation}
where $\overline \psi \in \mathcal{N}$, $p\leq \mbox{dim}\,(H/\mathcal{A}(\mathcal{N}))$,  $\gamma_i >0$ with $\sum_{i=1}^p \gamma_i = \phi(\overline u)$ and 
\begin{equation*}
u_i + \mathcal{N} \in \mbox{Ext}\left(\left\{u + \mathcal{N} \in X / \mathcal{N} : \phi(u) \leq 1\right\}\right)\,.
\end{equation*}
\end{thmnonumber}

Notice that our result completely characterizes the sparse solution $\overline{u}$ of \eqref{prointro} and relates the notion of \emph{sparsity} with structural properties of the regularizer $\phi$. Moreover, our hypotheses are minimal for having a well-posed variational problem \eqref{prointro}. 

The strategy to prove the previous theorem relies on the application of Krein--Milman's theorem and Carathéodory's theorem in the quotient space of $\mathcal{A}(X)$ that allows to represent any element in the image by $\A$ of the unit ball of the regularizer as a convex combination of the extremal points (see Theorem \ref{maint}). In order to prove minimality for the element having the desired representation, we derive optimality conditions for Problem \eqref{prointro} (Proposition \ref{opt}). For this purpose, we need to prove a \emph{no gap} property in the quotient space between primal and dual problem. In locally convex vector spaces this is not straightforward and requires the notion of Mackey topology \cite{Schaefer}.

\medskip
In the second part of our paper we apply the main result to specific examples of popular regularizers. First of all we recover the well-known result (see for example \cite{shapiro}) that by minimizing the Radon norm of a measure under finite-dimensional data constraints, one recovers a minimizer that is made of delta peaks. Indeed, according to our theory which applies when the space of Radon measures $\M(\Omega)$ is equipped with the weak* topology, Dirac deltas are extremal points of the unit ball associated with the Radon norm of a measure and our result applies straightforwardly (see Section \ref{radonnormmeas}). 

Then, we consider the TV regularizer for BV functions in bounded domains. Also in this case, our result applies when $BV(\Omega)$ is equipped with the weak* topology. This justifies the usage of locally convex spaces in the general theory. In order to confirm the heuristic observation that \emph{sparse} minimizers show a peculiar staircase effect, we characterize the extremal points of the unit ball associated to the TV norm (in the quotient space $BV(\Omega)/\R$). In particular, we extend a result of \cite{Ambrosiocasellesconnected} and \cite{Flemingext} to the case where $\Omega$ is a bounded domain. In order to achieve that, we need an alternative notion of simple sets of finite perimeter (see Definition \ref{simpleset}). We prove the following theorem:
\begin{thmnonumber}[Theorem~\ref{sparsity_tv}]
If $X = BV(\Omega)$ and $\phi(u) = |Du|(\Omega)$ there exists a minimizer $\overline u \in BV(\Omega)$ of \eqref{prointro} such that
\begin{equation}\label{nullintr}
\overline{u} = c + \sum_{i=1}^p \frac{\gamma_i}{P(E_i,\Omega)} \chi_{E_i} \, ,
\end{equation}
where $c \in \R$, $p\leq \mbox{dim}\,(H/\mathcal{A}(\R))$,  $\gamma_i >0$ with $\sum_i \gamma_i = |D\overline u|(\Omega)$ and $E_i \subset \Omega$ are simple sets with finite perimeter $P(E_i,\Omega)$ in $\Omega$.
\end{thmnonumber}
Finally, we apply our main result to the setting considered in \cite{unsersplines} and \cite{exactsolutionsflinth}, i.e., where the regularizer is given by $\phi(u) = \|Lu\|_\M$ for a scalar linear differential operator $L$. We remove the hypotheses concerning the structure of the null-space of $L$ and we work in the space of finite-order distributions equipped with the weak* topology. This allows us to have a general framework for these inverse problems that does not require additional assumptions on the Banach structure of the minimization domain (see \cite{chambollerepresenter} and \cite{exactsolutionsflinth} for comparison). It also justifies once more the usage of locally convex spaces in the abstract theory. In this setting, as an application of our main theorem, we are able to recover the same result as in \cite{unsersplines} and \cite{exactsolutionsflinth}.

\begin{thmnonumber}[Theorem~\ref{sparsity_diffop}]
Let $X = C_0^s(\Omega)^*$ (for $s$ sufficiently large, depending only on $L$ and $\Omega$) and $\phi(u) = \|Lu\|_{\mathcal{M}}$. Then, there exists $\overline u$ a minimizer of \eqref{prointro} such that
\begin{equation}\label{null0}
\overline{u} = \overline \psi + \sum_{i=1}^p \gamma_i G_{x_i} \, ,
\end{equation}
where $\overline\psi \in \Na = \{\psi \in C_0^s(\Omega)^*: L\psi = 0\}$, $p\leq \mbox{dim}\,(H/\A(\mathcal{N}))$, $x_1,\ldots, x_p \in \Omega$, $\gamma_1,\ldots,\gamma_p \in \R\setminus \{0\}$ with $\sum_i |\gamma_i| = \|L\overline u\|_{\mathcal{M}}$ (we denote by $G_{x}$ the fundamental solution of $L$ obtained by the Malgrange--Ehrenpreis theorem translated by $x$).
\end{thmnonumber}

\section{Setting and preliminary results}
\subsection{Basic assumptions on the functionals}\label{assumptions}
Let $(X,\tau)$ be a real locally convex space, i.e., the topology is
generated by a separating family of seminorms, and $(X^*,\tau')$ its
topological dual equipped with the weak* topology. Further, let $H$ be
an $N$-dimensional real Hilbert space and
$\mathcal{A} : X \rightarrow H$ a linear continuous operator and we
denote by $\mathcal{A}^* : H \rightarrow X^*$ its continuous adjoint,
defined thanks to Riesz's theorem as
\begin{equation*}
  \langle \mathcal{A}^* w , u\rangle = \langle w, \mathcal{A} u\rangle  
\end{equation*} 
for every $u\in X$ and $w \in H$.  Notice that we have denoted by
$\langle\cdot,\cdot \rangle$ both the scalar product in the Hilbert
space $H$ the duality product between $X$ and $X^*$.

As anticipated in the introduction we deal with a variational problem
of the type
\begin{equation}
\inf_{u\in X} \phi(u) + F(\A u)\,.
\end{equation}
In the remaining part of this section we describe the assumptions on $F$ and $\phi$ separately.

\smallskip

\textbf{- \emph{Assumptions on $F$}}:

\smallskip

We consider
\begin{equation*}
F: H \rightarrow (-\infty, +\infty] 
\end{equation*}
a proper convex function that is coercive, %
and lower semi-continuous with respect to the topology of $H$, which is
the standard topology on finite-dimensional spaces.

\smallskip

\textbf{- \emph{Assumptions on $\phi$}}:

\smallskip

We consider 
\begin{equation*}
\phi: X \rightarrow [0, +\infty]\, ,
\end{equation*}
a seminorm and that is lower semi-continuous with respect to the topology of $X$.
We make the following additional assumption:
\begin{enumerate}[label={[\textbf{H0}]}]
\item\label{H0} $\mathcal{A}(\mbox{dom}\, \phi) = H$\, ,
\end{enumerate}
where $\mbox{dom}\,\phi$ denotes the domain of $\phi$, i.e.
\begin{displaymath}
\mbox{dom}\,\phi = \{u\in X : \phi(u) < +\infty\}\,.
\end{displaymath}
Defining the null-space of $\phi$ as $\mathcal{N} =  \{u\in X: \phi(u)=0\}$, which is a closed subspace of $X$, we consider the following quotient space:
\begin{equation}
X_{\mathcal{N}} := X \Big{/}\mathcal{N}\, , 
\end{equation}
endowed with the quotient topology. It is well-known that
$X_{\mathcal{N}}$ is a locally convex space \cite{Rudin}.  We call
$\pi_{\mathcal{N}}: X \rightarrow X_{\mathcal{N}}$ the canonical
projection onto the quotient space and for simplifying the notation,
given $u\in X$ we denote by $u_{\mathcal{N}} = u + \mathcal{N}$ the
image of $u$ in the quotient space by $\pi_{\mathcal{N}}$. Likewise,
for $U \subset X$, we tacitly identify the Minkowski sum
$U + \Na \subset X$ with its image under $\pi_\Na$ in $X_\Na$.
 
Define then $\phi_{\mathcal{N}} : X_{\mathcal{N}} \rightarrow [0,+\infty]$ as  
\begin{equation}\label{mapquot}
\phi_{\mathcal{N}} (u_{\mathcal{N}}) := \phi(u) \,.
\end{equation}
Note that $\phi_{\mathcal{N}}$ is well-defined as $\phi$ is constant on the set $u + \mathcal{N}$ for every $u \in X$. Moreover, it is a seminorm in $X_{\mathcal{N}}$.

We assume that
\begin{enumerate}[label={[\textbf{H1}]}]
\item\label{H1} $\phi_{\mathcal{N}}$ is coercive, i.e. the sublevel sets
\begin{displaymath}
S^-(\phi_{\mathcal{N}},\alpha):=\{u_\mathcal{N} \in X_{\mathcal{N}} : \phi_{\mathcal{N}}(u_\mathcal{N})\leq \alpha\}
\end{displaymath}
are compact for every $\alpha > 0$.
\end{enumerate}
\begin{rmk}\label{coerclsc}
  Note that $\phi_\Na$ is lower semi-continuous in $X_\Na$: Indeed, as
  $\phi$ is lower semi-continuous, the superlevel-sets
  $S^+(\phi, \alpha) = \{ u \in X: \phi(u) > \alpha\}$ are open in $X$
  for each $\alpha$.  Now, as $\phi_\Na(u_\Na) > \alpha$ if and only
  if $\phi(u) > \alpha$, we have
  $S^+(\phi_\Na, \alpha) = \pi_\Na(S^+(\phi, \alpha))$. Since $\pi_N$
  is an open map ($X$ is a topological group with respect to
  addition), each $S^+(\phi_\Na, \alpha)$ is open in $X_\Na$ meaning
  that $\phi_\Na$ is lower semi-continuous.

  As a consequence, in order to obtain \ref{H1}, it suffices
  that each $S^-(\phi_\Na, \alpha)$ is contained in a compact set.
\end{rmk}
From now on we %
assume that $\A$, $F$ and $\phi$ satisfy the properties described
above.

\subsection{Existence of minimizers}

We state the following minimization problem:
\begin{pro}[Minimization problem in $X$]\label{problem}
Given $\phi$, $\mathcal{A}$ and $F$ with the assumptions given in the previous section, define for $u\in X$ the following functional:
\begin{equation}
J(u) := \phi(u) + F(\mathcal{A} u)\, .
\end{equation}
We aim at solving
\begin{equation*}
\min_{u \in X} J(u)\,.
\end{equation*}
\end{pro}

In order to prove the existence of minimizers for Problem \ref{problem} we state an auxiliary minimization problem in the quotient space $X_{\mathcal{N}}$.
\begin{pro}[Minimization problem in $X_{\mathcal{N}}$]\label{auxpro}
Given $F$, $\phi$ and $\mathcal{A}$ with the assumptions given in the previous section, we define
\begin{equation}
  \label{eq:auxpro}
\mathscr{J}(u_\mathcal{N}) = \phi_{\mathcal{N}}(u_{\mathcal{N}}) + \inf_{\psi \in \mathcal{N}} F(\mathcal{A} (u +\psi))\, .
\end{equation}
We want to solve
\begin{equation*}
\min_{u_{\mathcal{N}} \in X_{\mathcal{N}}} \mathscr{J}(u_{\mathcal{N}})\, .
\end{equation*}
\end{pro}
Note that the functional $\mathscr{J}$ is well-defined in $X_{\mathcal{N}}$ as both summands in~\eqref{eq:auxpro} are constant on $u + \mathcal{N}$ for every $u \in X$. 
We aim at proving existence of minimizers for Problem \ref{auxpro}. For this reason we firstly prove a lemma about the coercivity %
of functionals defined in quotient spaces.

\begin{lemma}\label{coercive_quotient}
  Let $Y$ be a locally convex space and
  $f: Y \rightarrow (-\infty , +\infty]$ be coercive.  Given
  $\mathcal{M} \subset Y$ a closed subspace of $Y$, we define,
  $\widetilde f : Y_\M \rightarrow (-\infty, +\infty]$ on the space
  $Y_\M = Y/\M$ as
  \begin{displaymath}
    \widetilde f(u_{\mathcal{M}}) = \inf_{v \in \mathcal{M}} f(u + v)\, .
  \end{displaymath}  
  Then, $\widetilde f$ is coercive with respect to the quotient
  topology of $Y_\M$.
\end{lemma}

\begin{proof}
  By coercivity, the sublevel sets $S^-(f, \alpha)$ are compact for
  each $\alpha$. Since the projection $\pi_{\M}$ is continuous,
  each $\pi_\M(S^-(f, \alpha))$ is compact in $Y_\M$. Now, 
  \[ \pi_\M(S^-(f,\alpha)) = \{u_\M \in Y_\M: \text{there exists} \
    v \in \M \ \text{such that} \ f(u + v) \leq \alpha \}.
  \]
  Since, by definition, $\widetilde f(u_\M) \leq \alpha$ if and only
  if for each $\varepsilon > 0$ there exists $v \in \M$ such that
  $f(u+v) \leq \alpha + \varepsilon$, the identity
  $S^-(\widetilde f,\alpha) = \bigcap_{\varepsilon > 0}
  \pi_\M(S^-(f,\alpha+\varepsilon))$ follows. The right-hand side is
  compact as an intersection of compact sets, hence each
  $S^-(\widetilde f,\alpha)$ is compact, showing the coercivity of
  $\tilde f$.
\end{proof}

\begin{prop}\label{auxi}
There exists a minimizer for Problem \ref{auxpro}.
\end{prop}
\begin{proof}
  As $F$ is %
  proper, using Hypothesis \ref{H0} we infer that the infimum
  of Problem~\ref{auxpro} is not %
  $+\infty$. Likewise, since $F$ is convex, lower semi-continuous and
  coercive, it is bounded from below such that the infimum of
  Problem~\ref{auxpro} is also not $-\infty$. Let us show that the
  proper and convex function
  $u \mapsto \inf_{\psi \in \Na} F(\A(u + \psi))$ is lower
  semi-continuous in $X$. For that purpose, observe that $\A(\Na)$ is
  a subspace of the finite-dimensional space $H$ and hence
  closed. Denote by $H_\Na$ the quotient space $H/\A(\Na)$ on which we
  define $F_\Na: H_\Na \to (-\infty,+\infty]$ according to
  \[
    F_\Na(w_\Na) = \inf_{\eta \in \A(\Na)} F(w + \eta)
  \]
  where $w_\Na = w + \A(\Na)$.  Note that this functional is well-defined on $H_\mathcal{N}$ as given $w^1, w^2 \in H$ with $w^1 - w^2 \in \mathcal{A}(\mathcal{N})$ there holds
  
  \begin{equation*}
  \inf_{\eta \in \A(\Na)} F(w^1 + \eta) = \inf_{\eta \in \A(\Na)} F(w^2 + \eta)\,.
  \end{equation*}
 Moreover, $F_\Na$ is proper
  and convex.  As $F$ is assumed to be coercive, applying
  Lemma~\ref{coercive_quotient} yields that $F_\Na$ is also coercive
  and lower semi-continuous in particular. Now,
  \[
    \inf_{\psi \in \Na} F(\A(u + \psi)) = (F_\Na \circ
    \pi_{\A(\Na)} \circ \A)(u)
  \]
  where the right-hand side is a composition of continuous linear maps
  and a lower semi-continuous functional and hence, lower
  semi-continuous. Obviously, replacing $u$ by $u + \varphi$,
  $\varphi \in \Na$ does not change the value of this functional, so
  by the same argument as in Remark~\ref{coerclsc}, we deduce that
  $u_\Na \mapsto \inf_{\psi \in \Na} F(\A(u + \psi))$ and
  consequently $\mathscr{J}$, is lower semi-continuous.

Notice now that %
\begin{displaymath}
S^-(\mathscr{J},\alpha) := \{u_{\mathcal{N}} \in X_{\mathcal{N}}: \mathscr{J}(u_{\mathcal{N}}) \leq \alpha\} \subset S^-(\phi_\mathcal{N},\alpha - \inf F) \, .
\end{displaymath}
Therefore, as $\mathscr{J}$ is lower semi-continuous and
$\phi_{\mathcal{N}}$ is coercive due to Hypothesis \ref{H1}, we
infer that $S^-(\mathscr{J},\alpha)$ is compact for every $\alpha \in \R$.

We want to prove that $\mathscr{J}$ admits a minimizer in $X_{\mathcal{N}}$. Notice that the collection $\{S^-(\mathscr{J},\alpha)\}_{\alpha>\inf \mathscr{J}}$ has the finite intersection property. As the set $S^-(\mathscr{J},\alpha_0)$ is compact for an arbitrary $\alpha_0 > \inf \mathscr{J}$ and each $S^-(\mathscr{J},\alpha)$ is closed, we infer that 
\begin{displaymath}
\bigcap_{\inf \mathscr{J} <\alpha \leq \alpha_0} S^-(\mathscr{J},\alpha) \neq \emptyset\, .
\end{displaymath}
Choosing $\overline u_\mathcal{N} \in \bigcap_{\inf \mathscr{J} <\alpha\leq \alpha_0} S^-(\mathscr{J},\alpha)$ we notice that it is a minimizer of $\mathscr{J}$ as
\begin{displaymath}
\mathscr{J}(\overline{u}_\mathcal{N}) \leq \inf_{u_{\mathcal{N}} \in X_{\mathcal{N}}} \mathscr{J}(u_{\mathcal{N}})\,. \qedhere
\end{displaymath}
\end{proof}
We are now in position to prove the existence of minimizers for Problem \ref{problem}.
\begin{thm}\label{ex}
Given $\overline u_\mathcal{N} = \overline u + \mathcal{N}$ a minimizer for Problem \ref{auxpro}, there exists $\overline \psi \in \mathcal{N}$ such that $\overline u + \overline \psi$ is a minimizer for Problem \ref{problem}.
\end{thm}
\begin{proof}
Notice that 
for every $\psi \in \mathcal{N}$ we have
\begin{equation}\label{equal}
\inf_{u \in X} \phi(u) + F(\mathcal{A} u) = \inf_{u\in X} \phi(u) + F(\mathcal{A} (u+ \psi))\,.
\end{equation}
Hence taking the infimum with respect to $\psi \in \mathcal{N}$ on both sides we obtain that Problem \ref{problem} and Problem \ref{auxpro} have the same infimum. 
Let $\overline u_\mathcal{N}$ be a minimizer for Problem \ref{auxpro}. Then consider the following minimization problem:
\begin{displaymath}
\inf_{\eta \in \mathcal{A}(\mathcal{N})} F(\mathcal{A}\overline u +  \eta)\, .
\end{displaymath}
As $F$ is proper, convex, lower semi-continuous and coercive as well as
$\A(\Na)$ is finite-dimensional and hence closed in $H$, %
the infimum is realized and finite. Denoting by $\overline \eta$ a minimizer, we choose $\overline \psi \in \mathcal{N}$ such that $\mathcal{A} \overline \psi = \overline\eta$. Then, %
$\overline v := \overline u + \overline \psi$ is a minimizer for Problem \ref{problem}. Indeed,
\begin{eqnarray*}
\phi(\overline v) + F(\mathcal{A} \overline v) &=& \phi(\overline u) +  \inf_{\eta \in \mathcal{A}(\mathcal{N})} F(\mathcal{A} \overline u +  \eta)  \\
&=& \phi(\overline u) +  \inf_{\psi \in \mathcal{N}} F(\mathcal{A}(\overline u + \psi))\\
 & = & \mathscr{J}(\overline u_\mathcal{N})\, .
\end{eqnarray*}
Then, as the two minimization problems have equal infimum, we conclude.
\end{proof}

\begin{rmk}
The converse of Theorem \ref{ex} holds true. Namely, if $\overline u$ a minimizer for Problem \ref{problem}, then $\overline u_\mathcal{N} = \overline u + \mathcal{N}$ is a minimizer of Problem \ref{auxpro}. Indeed, for every $v_\mathcal{N} = v + \mathcal{N}$ and denoting by $\overline \psi$ a minimizer of $\psi \mapsto F(\mathcal{A}(v + \psi))$ in $\mathcal{N}$ (that exists for similar arguments as in the previous proof) we have
\begin{align*}
\mathscr{J}(\overline u_\mathcal{N}) \leq \phi(\overline{u}) + F(\mathcal{A}\overline{u}) \leq \phi(v + \overline \psi) +  F(\mathcal{A}(v + \overline \psi))  = \phi_\mathcal{N}(v_\mathcal{N}) + \inf_{\psi \in \mathcal{N}} F(\mathcal{A} (v +\psi))\,.
\end{align*}
\end{rmk}

\subsection{Optimality conditions}
In this section we want to obtain optimality conditions for Problem \ref{auxpro} deriving a dual formulation and showing that under our hypotheses we have \emph{no gap} between the primal and the dual problem. 

\medskip
In order to perform this analysis we need to endow the space $X_{\mathcal{N}}^*$ equipped with the weak*-topology with the associated Mackey topology. For the reader's convenience we remind the definition of the Mackey topology and we refer to \cite{Schaefer} for a comprehensive treatment. Given a real locally convex space $Y$, define the following family of seminorms on $Y^*$: %
\begin{equation}
\rho_A(u^*) = \sup\{|\langle u, u^*\rangle| : u \in A\} 
\end{equation}
for every $A\subset Y$ %
absolutely convex and weakly compact. This family of seminorms
generates a locally convex topology on $Y^*$ %
that is called \emph{Mackey topology} and it is denoted by
$\tau(Y^*,Y)$. It is the strongest topology on $Y^*$ such that
$Y$ is still the dual of $Y^*$  (see Theorem 9 in Section A.4 of \cite{Aubin}).

Further, we need the notion of \emph{Fenchel conjugate} functionals
which are defined as follows.  Given a real locally convex space
$Y$ and a proper function $f: Y \rightarrow
(-\infty,+\infty]$ we denote by $f^* : Y^* \rightarrow
(-\infty,+\infty]$ the conjugate of $f$ defined as
\begin{displaymath}
f^*(x^*) = \sup_{x\in Y} \big[\langle x^*, x\rangle - f(x)\big]\,.
\end{displaymath}

\medskip

In order to obtain the optimality conditions we will use the following well-known proposition (see Proposition 5 in Section 3.4.3 of \cite{Aubin}).

\begin{prop}\label{mackey}
Let $Y$ be a real locally convex space. Given a proper, lower semi-continuous, convex function $f:Y \rightarrow (-\infty,+\infty]$, the following statements are equivalent:
\begin{enumerate}[label=\roman*)]
\item $f^*$ is continuous in zero for the Mackey topology $\tau(Y^*,Y)$.
\item for every $\alpha \in \R$, the sublevel-set
\begin{displaymath}
  S^-(f,\alpha) := \{x \in Y: f(x) \leq \alpha \}
\end{displaymath}
is compact with respect to the weak topology.
\end{enumerate}
\end{prop}
\begin{rmk}
In the next proposition, we will apply this result for $f=\phi_\mathcal{N}$, a proper and lower semi-continuous seminorm. In this case, the proof of Proposition \ref{mackey} is straightforward. Indeed, $\phi_\mathcal{N}^* = I_{\{\rho_S(u^*) \leq 1\}}$ where $I$ is the indicator function and $S =\{u: \phi_\mathcal{N}(u) \leq 1\}$. Hence, if $S$ is weakly compact, then thanks to the definition of the Mackey topology, $\phi^*_\mathcal{N}$ is continuous in zero.

Conversely, if $\phi_\Na^*$ is continuous in zero, then there exist
absolutely convex, weakly compact sets $A_1,\ldots,A_n \subset X_\Na$
and $\varepsilon_1, \ldots, \varepsilon_n > 0$ such that
$\rho_{A_i}(u^*) \leq \varepsilon_i$ for $i=1,\ldots,n$ implies
$\rho_S(u^*) \leq 1$. This, however, means that
$S \subset \varepsilon_1^{-1} A_1 + \ldots + \varepsilon_n^{-1} A_n$.
Indeed, if this were not the case, one could separate a $u \in S$ from
the absolutely convex and weakly compact set
$\varepsilon_1^{-1} A_1 + \ldots + \varepsilon_n^{-1} A_n$ by a
$u^* \in X_\Na^*$ such that $\langle u^*, u\rangle > 1$ as well as
$\langle u^*, \sum_{i=1}^n \varepsilon_i^{-1} u_i \rangle \leq 1$ for
$u_i \in A_i$. In particular, $\rho_{A_i}(u^*) \leq \varepsilon_i$ for
each $i=1,\ldots,n$ leading to the contradiction
$\rho_{S}(u^*) \leq 1$. Due to lower semi-continuity of $\phi_\Na$,
$S$ is a closed convex subset of a weakly compact set and hence weakly compact. By
positive homogeneity of $\phi_\Na$, the sets $S^-(\phi_\Na, \alpha)$
are compact for all $\alpha \in \R$.
\end{rmk}

For the following, it is convenient to define the linear operator $\mathcal{A}_\mathcal{N}: X_{\mathcal{N}} \rightarrow H_\mathcal{N}:= H / \mathcal{A}(\mathcal{N})$ as
\begin{equation}
\mathcal{A}_\mathcal{N}u_{\mathcal{N}} = \mathcal{A} u + \mathcal{A}(\mathcal{N})\, .
\end{equation}

\begin{rmk}\label{continuity}
Notice that $\mathcal{A}_\mathcal{N}$ is well-defined in $X_\mathcal{N}$ as given $u^1, u^2 \in X$ with $u^1 - u^2 \in \mathcal{N}$  it also holds that $\mathcal{A}u^1 - \mathcal{A}u^2 \in \mathcal{A}(\mathcal{N})$. Moreover, it is continuous in $X_{\mathcal{N}}$; indeed, the following diagram commutes:
\begin{equation*}
\xymatrix{
X  \ar[d]_{\pi_{\mathcal{N}}} \ar[r]^{\mathcal{A}} & H \ar[d]^{\pi_{\A(\Na)}}\\
X_{\mathcal{N}}  \ar[r]^{\mathcal{A}_\mathcal{N}} & H_\mathcal{N}}
\end{equation*} 
and the projections on the quotients are continuous and open.
\end{rmk}
We denote by $\mathcal{A}_\mathcal{N}^*: H_\mathcal{N} \rightarrow X_{\mathcal{N}}^*$ its adjoint that has finite-dimensional image and is hence continuous for each topology that makes $X_{\mathcal{N}}^*$ a topological vector space.
Given $w\in H$, we denote by $w_{\mathcal{N}} := w + \A(\mathcal{N})$
an element of $H_\mathcal{N}$.

We can equivalently write
\begin{equation}
\mathscr{J}(u_\mathcal{N}) = \phi_\mathcal{N}(u_{\mathcal{N}}) + F_\mathcal{N}(\mathcal{A}_\mathcal{N}(u_{\mathcal{N}}))\, ,
\end{equation}
where
\begin{displaymath}
F_\mathcal{N}(w_\mathcal{N}) = \inf_{\eta \in \mathcal{A}(\mathcal{N})} F(w + \eta)\,. 
\end{displaymath}
\begin{rmk}\label{plc}
Notice again that $F_\mathcal{N}$ is proper, convex and, applying Lemma~\ref{coercive_quotient} %
with $f = F$ and $\mathcal{M} = \mathcal{A}(\mathcal{N})$, it is also 
coercive in $H_\mathcal{N}$.
\end{rmk}
We now derive %
optimality conditions for Problem~\ref{auxpro}. For that purpose,
recall that given a functional $f: Y \to (-\infty,+\infty]$ on a real
locally convex space $Y$, the element $x^* \in Y^*$ is called a 
\emph{subgradient} of $f$ in $x \in Y$, if 
\[
  f(x) + \langle x^*, y - x \rangle \leq f(y)
\]
for each $y \in Y$. In this case, we denote $x^* \in \partial f(x)$.

\begin{prop}[Optimality conditions]\label{opt}
It holds that $\overline{u}_\mathcal{N}\in X_\mathcal{N}$ is a minimizer for Problem \ref{auxpro} if and only if there exists $\overline{w}_{\mathcal{N}} \in H_\mathcal{N}$ 
such that
\begin{enumerate}[label=\roman*)]
\item\label{opt_i} $\mathcal{A}_\mathcal{N}^* \overline{w}_{\mathcal{N}} \in \partial \phi_{\mathcal{N}} (\overline{u}_\mathcal{N})$,
\item\label{opt_ii} $\mathcal{A}_\mathcal{N} \overline{u}_\mathcal{N} \in \partial F_\mathcal{N}^* (-\overline{w}_{\mathcal{N}})$.
\end{enumerate}
\end{prop}

\begin{proof}
  We start with transforming Problem \ref{auxpro} into the problem
  \eqref{dual_problem} for which the dual problem in terms of
  Fenchel--Rockafellar duality (see Remark III.4.2 in \cite{CAAV}) will
  turn out to be equivalent to the original problem:
  \begin{equation}
    \label{dual_problem}\tag{$\mathscr{P}^*$}
    \inf_{w_\mathcal{N} \in H_\mathcal{N}} \Big[\phi_{\mathcal{N}}^*(\mathcal{A}_\mathcal{N}^* w_\mathcal{N}) + F_\mathcal{N}^* (-w_\mathcal{N}) \Big].
  \end{equation}
  We now endow $X_\mathcal{N}^*$ with the Mackey topology $\tau(X^*_\mathcal{N}, X_\mathcal{N})$. 
  Notice that $\phi_{\mathcal{N}}^*$ is convex, proper and weakly* lower semi-continuous and hence it is lower semi-continuous with respect to $\tau(X^*_\mathcal{N}, X_\mathcal{N})$ as well. As previously mentioned, the adjoint $\mathcal{A}_\mathcal{N}^*: H_\mathcal{N} \rightarrow X_\mathcal{N}^*$ is linear and continuous with respect to any vector space topology and in particular, the weak* topology of $X^*_\mathcal{N}$ as well as the Mackey topology $\tau(X^*_\mathcal{N}, X_\mathcal{N})$. %
  Moreover, thanks to Remark \ref{plc}, $F_\mathcal{N}^*$ is convex, proper and lower semi-continuous in $H_\mathcal{N}$.
  
  Notice that as $\phi_{\mathcal{N}}$ satisfies Hypothesis
  \ref{H1} (that implies in particular that the sublevel sets
  of $\phi_{\mathcal{N}}$ are weakly compact), using
  Proposition~\ref{mackey}, we have that $\phi^*_\mathcal{N}$ is
  continuous in
  zero. %
  Hence, applying Theorem III.4.1 in \cite{CAAV}, the
  problem~\eqref{dual_problem} has zero gap to its dual which
  coincides, as the dual space of $X_\Na^*$ is $X_\Na$ and
  $\phi_\Na^{**} = \phi_\Na$ as well as $F_\Na^{**} = F_\Na$, with
  Problem~\ref{auxpro}, i.e.,
  \begin{displaymath}
    \inf_{u_\mathcal{N} \in X_\mathcal{N}} \mathscr{J} (u_\mathcal{N}) = 
    \inf_{u_\Na \in X_\Na} \Big[ \phi_\Na(u_\Na) + F_\Na(\A_\Na u_\Na) \Big] =
    -\inf_{w_\mathcal{N} \in H_\mathcal{N}} \Big[\phi_{\mathcal{N}}^*(\mathcal{A}_\mathcal{N}^* w_\mathcal{N}) + F_\mathcal{N}^* (-w_\mathcal{N})\Big] \,.
  \end{displaymath}
  In order to establish the optimality conditions, we want to prove now that the problem~\eqref{dual_problem} has a minimizer, since the existence of a minimizer for Problem~\ref{auxpro} has already been established in Theorem~\ref{auxi}. Notice that at this point there is no more need to consider the Mackey topology on $X^*_\mathcal{N}$ and we can use the weak* topology on $X^*_\mathcal{N}$.

The functional $\phi_{\mathcal{N}}^* \circ \mathcal{A}_\mathcal{N}^* + F_\mathcal{N}^* \circ (-\id)$ is convex, proper and lower semi-continuous.
We aim at showing that it is also coercive. It is enough to prove that $\phi_{\mathcal{N}}^* \circ \mathcal{A}_\mathcal{N}^*$ is the indicator function of a compact convex set as $F_\Na^*$ is proper, convex and lower semi-continuous.

Notice that
\begin{eqnarray}
(\phi_{\mathcal{N}}^* \circ \mathcal{A}_\mathcal{N}^*)(w_\mathcal{N}) &=& \sup_{u_\mathcal{N} \in X_\mathcal{N}} \langle u_\mathcal{N}, \mathcal{A}_\mathcal{N}^* w_\mathcal{N}\rangle - \phi_{\mathcal{N}}(u_\mathcal{N}) \nonumber \\
&=& \sup_{u_\mathcal{N} \in X_\mathcal{N}}  \langle \mathcal{A}_\mathcal{N} u_\mathcal{N},   w_\mathcal{N}\rangle - \phi_{\mathcal{N}}(u_\mathcal{N}) \nonumber \\
& = & \sup_{v_\mathcal{N}\in H_\mathcal{N}, \mathcal{A}_\mathcal{N} u_\mathcal{N} = v_\mathcal{N}}  \langle v_\mathcal{N},   w_\mathcal{N}\rangle - \phi_{\mathcal{N}}(u_\mathcal{N}) \nonumber \\
& = & \sup_{v_\mathcal{N}\in H_\mathcal{N}} \Big[ \langle v_\mathcal{N},   w_\mathcal{N}\rangle - \inf_{\mathcal{A}_\mathcal{N} u_\mathcal{N} = v_\mathcal{N}} \phi_{\mathcal{N}}(u_\mathcal{N})\Big] \nonumber \\
& = & \left[v_\Na \mapsto \inf_{\mathcal{A}_{\mathcal{N}}u_\mathcal{N} = v_\mathcal{N}}\phi_{\mathcal{N}}(u_\mathcal{N})\right]^*(w_\Na) \label{double} \, .
\end{eqnarray}
We prove that the map $G:H_\mathcal{N}\rightarrow (-\infty,+\infty]$ defined as
\begin{displaymath}
G(v_\mathcal{N}) = \inf_{\mathcal{A}_{\mathcal{N}}u_\mathcal{N} = v_\mathcal{N}}\phi_{\mathcal{N}}(u_\mathcal{N})
\end{displaymath}
is convex, proper and coercive %
in $H_\mathcal{N}$. For this purpose, notice that, as a consequence of Hypothesis \ref{H0} and the definition of $\mathcal{A}_\mathcal{N}$, we have
\begin{displaymath}
\mathcal{A}_\mathcal{N}(\mbox{dom}\, \phi_\mathcal{N}) = \mathcal{A}_\mathcal{N} (\mbox{dom}\, \phi + \mathcal{N}) = \mathcal{A}( \mbox{dom}\,\phi) + \mathcal{A}(\mathcal{N}) = H + \mathcal{A}(\mathcal{N}) = H_\mathcal{N} \,.
\end{displaymath}
Hence, $\mbox{dom}\, G = H_\mathcal{N}$ and $\A_\Na$ is
surjective. With $\A_0 : X_\Na/\ker(\A_\Na) \to H_\Na$ defined via
$\A_0(u_\Na + \ker(A_\Na)) = \A_\Na u_\Na$ which is bijective and
hence, continuously invertible, we can write
$G(v_\Na) = \phi_0(\A_0^{-1} v_\Na)$ where
\[
  \phi_0(u_\Na + \ker(\A_\Na)) = \inf_{\psi_\Na \in \ker(\A_\Na)}
  \phi_\Na(u_\Na + \psi_\Na)\,.
\]
By Hypothesis \ref{H1}, $\phi_\Na$ is coercive, so
Lemma~\ref{coercive_quotient} yields that $\phi_0$ is coercive. As
$\phi_\Na$ is a seminorm, $\phi_0$ is proper and convex. It follows
that $G$ is proper, convex and lower semi-continuous.  As
$\mbox{dom}\, G = H_\mathcal{N}$, convexity implies that $G$ is
continuous everywhere in $H_\Na$ and in particular, in
zero. Consequently, $G^* = \phi_\Na^* \circ \A_\Na^*$ is coercive.
It follows that $G^*$ is the indicator function of a compact convex set as $G$ is one-homogeneous.
Hence, applying the direct method of calculus of variations in $H_\mathcal{N}$ we infer that Problem~\eqref{dual_problem} has a minimizer that we denote by $\overline{w}_{\mathcal{N}} \in H_\mathcal{N}$. 

We now derive the optimality conditions for this problem.
Applying Proposition III.4.1 in \cite{CAAV} we infer that
\begin{displaymath}
\phi_\mathcal{N}(\overline u_\mathcal{N}) + F_\mathcal{N}(\mathcal{A}_\mathcal{N} \overline u_\mathcal{N})+ \phi^*_\mathcal{N}(\mathcal{A}^*_\mathcal{N} \overline w_\mathcal{N}) + F^*_\mathcal{N}(-\overline w_\mathcal{N})= 0 = \langle  \overline u_\mathcal{N} , \mathcal{A}^*_\mathcal{N} \overline w_\mathcal{N}\rangle + \langle \mathcal{A}_\mathcal{N} \overline u_\mathcal{N}, - \overline w_\mathcal{N}\rangle\,.
\end{displaymath}
Therefore, we get with the help of the Fenchel inequality that
\begin{equation}\label{fas}
\phi_\mathcal{N}(\overline u_\mathcal{N}) + \phi^*_\mathcal{N}(\mathcal{A}^*_\mathcal{N} \overline w_\mathcal{N}) - \langle  \overline u_\mathcal{N} , \mathcal{A}^*_\mathcal{N} \overline w_\mathcal{N}\rangle = 0  \qquad \mbox{ and }
\end{equation}
\begin{equation}\label{fas2}
F_\mathcal{N}(\mathcal{A}_\mathcal{N} \overline u_\mathcal{N})  + F^*_\mathcal{N}(-\overline w_\mathcal{N}) + \langle \mathcal{A}_\mathcal{N} \overline u_\mathcal{N},  \overline w_\mathcal{N}\rangle = 0\, .
\end{equation}
Finally, Equations \eqref{fas} and \eqref{fas2} are equivalent to
\begin{equation}\label{other}
\mathcal{A}_\mathcal{N}^* \overline{w}_{\mathcal{N}} \in \partial \phi_{\mathcal{N}} (\overline{u}_\mathcal{N}) \quad \mbox{ and  } \ \ 
\mathcal{A}_\mathcal{N} \overline u_\mathcal{N} \in \partial F_\mathcal{N}^* (-\overline w_{\mathcal{N}})
\end{equation}
as we wanted to prove.

Vice versa, if there exist $\overline{w}_{\mathcal{N}}$ and $\overline u_\mathcal{N}$ that satisfy the optimality conditions \ref{opt_i} and \ref{opt_ii}, applying again Proposition III.4.1 in \cite{CAAV} we deduce that $\overline u_\mathcal{N}$ is a minimizer of Problem \ref{auxpro} and 
$\overline{w}_{\mathcal{N}}$ is a minimizer of~\eqref{dual_problem}.
\end{proof}

\begin{rmk}\label{newopt}
Defining the set
\begin{equation*}
\mathcal{K}:= \{u_\mathcal{N}^* \in X_\mathcal{N}^* : \langle u_\mathcal{N}^* , u_\mathcal{N} \rangle \leq \phi_\mathcal{N}(u_\mathcal{N}) \mbox{ for every } u_\mathcal{N} \in X_\mathcal{N}\}\, ,
\end{equation*}
condition \ref{opt_i} of Proposition \ref{opt} is equivalent to
\begin{enumerate}[label=\roman*)]
\item\label{alt_opt_i} $\mathcal{A}_\mathcal{N}^* \overline w_{\mathcal{N}} \in \mathcal{K}$,
\item\label{alt_opt_ii} $\langle \mathcal{A}_\mathcal{N}^* \overline w_{\mathcal{N}}, \overline u_\mathcal{N}\rangle = \phi_{\mathcal{N}}(\overline u_\mathcal{N})$.
\end{enumerate}
\end{rmk}

\section{Abstract main result: existence of a sparse minimizer}

Define
\begin{equation*}
B := \{u \in X : \phi(u) \leq 1\}\,
\end{equation*}
and $B_\mathcal{N} := B + \mathcal{N} \subset X_\Na$.
\begin{defi}[Extremal points]
Given a convex set $K$ of a locally convex space we define the extremal points of $K$ as the points $k\in K$ such that if there exists $t \in (0,1)$, $k_1,k_2 \in K$ such that
\begin{displaymath}
k = t k_1 + (1-t)k_2\, ,
\end{displaymath}
then $k=k_1=k_2$.

The set of extremal points of $K$ will be denoted by $Ext(K)$.
\end{defi}

First we need a lemma about the behaviour of extremal points under a linear mapping.

\begin{lemma}\label{fund}
Let $K$ be a convex set in a locally convex space $X$. Given $Y$ a real topological vector space and a linear map $L:X \rightarrow Y$ the following statements hold:
\begin{enumerate}[label=\roman*)]
\item\label{fund_i} If $L$ is continuous and $K$ is compact, then $Ext(LK) \subset LExt(K)$. 
\item\label{fund_ii} If $L$ is injective, then $Ext(LK) = LExt(K)$\,.
\end{enumerate} 
\end{lemma}
\begin{proof}
To prove \ref{fund_i} let us consider $k \in K$ such that $Lk$ is an extremal point of $LK$. We want to show that there exists $\overline k \in Ext(K)$ such that $Lk = L\overline k$ which proves the first claim.

Consider the set $(k + \ker L) \cap K$. As this is a non-empty compact convex set in a locally convex space ($\ker L$ is closed by the continuity of $L$), by the Krein--Milman theorem, it admits an extremal point denoted by $\overline k \in (k + \ker L) \cap K$. In order to conclude the proof we need to prove that $\overline k \in Ext(K)$.
Assume the convex combination
\begin{equation}\label{ah}
\overline k = t k_1 + (1-t) k_2\, 
\end{equation}
for $k_1,k_2 \in K$ and $t\in (0,1)$.
Then applying the linear operator $L$ we obtain that
\begin{equation*}
L\overline k = t Lk_1 + (1-t) Lk_2\, .
\end{equation*}
As $Lk \in Ext(LK)$ and $L\overline k = Lk$ we infer that $L\overline k = Lk_1 = Lk_2$ and so $k_1,k_2 \in (k + \ker L) \cap K$. From \eqref{ah} and the extremality of $\overline k$ it follows that $\overline k = k_1= k_2$. 

Let us prove \ref{fund_ii}. To show that $ Ext(LK) \subset LExt(K)$ take $Lk \in Ext(LK)$ and assume the convex combination
\begin{equation*}
k = t k_1 + (1-t) k_2\, 
\end{equation*}
for $k_1, k_2 \in K$, $t\in (0,1)$. Applying $L$ to both sides and using that $Lk \in Ext(LK)$ we obtain that $Lk = Lk_1 = Lk_2$. Then the injectivity of $L$ implies that $k=k_1=k_2$, thus $k \in Ext(K)$.

To prove the opposite inclusion let us consider $k \in Ext(K)$. 
Assume the convex combination
\begin{equation*}
Lk = t Lk_1 + (1-t) Lk_2\, 
\end{equation*}
for $k_1,k_2 \in K$ and $t\in (0,1)$. As $L$ is injective  and using that $k \in Ext(K)$ we conclude that $k = k_1= k_2$ and hence $Lk = Lk_1= Lk_2$. 
\end{proof}

We are now in the position to prove our main theorem.

\begin{thm}\label{maint}
There exists $\overline u \in X$, a minimizer of Problem \ref{problem} 
with the representation:
\begin{equation}\label{null}
\overline{u} = \overline \psi + \sum_{i=1}^p \gamma_i u_i \, ,
\end{equation}
where $\overline \psi \in \mathcal{N}$, 
$p\leq \mbox{dim}\,H_\mathcal{N}$, $u_i + \mathcal{N} \in \mbox{Ext}(B_\mathcal{N})$ and $\gamma_i >0$ with $\sum_{i=1}^p \gamma_i = \phi(\overline u)$.
\end{thm}

\begin{proof}
We apply Proposition \ref{auxi} and Proposition \ref{opt} to find $\overline u_\mathcal{N} \in X_\mathcal{N}$ a minimizer of Problem \ref{auxpro} and $\overline w_{\mathcal{N}} \in H_\mathcal{N}$ such that properties \ref{opt_i} and \ref{opt_ii} in Proposition \ref{opt} hold. If $\phi_{\mathcal{N}}(\overline u_\mathcal{N}) = 0$, then 
applying Theorem \ref{ex}, we infer that there exists $\overline \psi \in \mathcal{N}$ such that $\overline \psi$ is a minimizer of Problem \ref{problem}. Therefore, Equation \eqref{null} holds with $p=0$.
Hence, we suppose without loss of generality that $\overline u_\Na \notin \mathcal{N}$, i.e. $\phi_\Na(\overline u_\Na) > 0$. 

Notice that
$\mathcal{A}_\mathcal{N} B_\mathcal{N} = \mathcal{A} B + \mathcal{A}(\mathcal{N}) \subset H_\mathcal{N}$ and
\begin{equation}
B_\mathcal{N} = \{u_{\mathcal{N}}\in X_\mathcal{N} : \phi_{\mathcal{N}}(u_{\mathcal{N}}) \leq 1\} \, .
\end{equation}
Hence, using Hypothesis \ref{H1} we infer that $B_\mathcal{N}$ is compact and thanks to Remark \ref{continuity} we have that $\mathcal{A}_\mathcal{N} B_\mathcal{N}$ is compact in $H_\mathcal{N}$ as well.
As $\frac{1}{\phi_\mathcal{N}(\overline u_\mathcal{N})}\mathcal{A}_\mathcal{N} \left(u_\mathcal{N}\right) \in \mathcal{A}_\mathcal{N} B_\mathcal{N} \subset H_\mathcal{N}$,
by the Krein--Milman theorem and Carath\'eodory theorem, we have that 
\begin{equation}\label{beginn}
\frac{1}{\phi_\mathcal{N}(\overline u_\mathcal{N})}\mathcal{A}_\mathcal{N} \left(\overline u_\mathcal{N}\right) = \sum_{i=1}^p \alpha_i w_i\, ,
\end{equation}
where $\alpha_i > 0$ for $i=1,\ldots,p$, $\sum_{i=1}^p \alpha_i = 1$, $w_i \in Ext(\mathcal{A}_\mathcal{N} B_\mathcal{N})$ and $p \leq \dim H_\mathcal{N} +1$. We can assume that $p$ is minimal, in the sense that it is the minimal number such that a decomposition like \eqref{beginn} holds.

Thanks to part \ref{fund_i} in Lemma \ref{fund} we have that there exist $v_i \in Ext(B_\mathcal{N})$ such that
\begin{equation}\label{representation0}
\frac{1}{\phi_\mathcal{N}(\overline u_\mathcal{N})}\mathcal{A}_\mathcal{N} \left(\overline u_\mathcal{N}\right) = \sum_{i=1}^p \alpha_i \mathcal{A}_\mathcal{N} v_i\, .
\end{equation}
We want to prove that $p\leq \dim H_\mathcal{N}$.
We claim that for every $i=1,\ldots,p$ we have that
\begin{equation}\label{claim}
\langle \mathcal{A}_\mathcal{N} v_i,  \overline w_{\mathcal{N}} \rangle = 1\, . 
\end{equation}
Indeed, thanks to Remark \ref{newopt} we have
\begin{eqnarray*}
1 &=& \langle \mathcal{A}_\mathcal{N}^* \overline w_{\mathcal{N}} , \frac{\overline u_\mathcal{N}}{\phi_\mathcal{N}(\overline u_\Na)}\rangle = \langle \overline w_{\mathcal{N}} , \frac{\mathcal{A}_\mathcal{N}(\overline u_\mathcal{N})} {\phi_\mathcal{N}(\overline u_\Na)}\rangle \\
 &=&  \sum_{i=1}^p \alpha_i \langle  \overline w_{\mathcal{N}} , \mathcal{A}_\mathcal{N} v_i\rangle \, .
\end{eqnarray*}
Moreover for all $u_\mathcal{N} \in B_\mathcal{N}$ we have $\langle \mathcal{A}_\mathcal{N} u_\mathcal{N} , \overline w_{\mathcal{N}}  \rangle = \langle u_\mathcal{N} , \mathcal{A}_\mathcal{N}^* \overline w_{\mathcal{N}}  \rangle  \leq \phi_{\mathcal{N}}(u_\mathcal{N}) \leq 1 $. Therefore,
as $v_i \in B_\mathcal{N}$ for every $i=1,\dots,p$, %
the claim stated in Equation \eqref{claim} follows.
Hence, $\{\mathcal{A}_\mathcal{N} v_i\}_i$ is contained in a $\dim H_\mathcal{N}-1$ dimensional Hilbert space (obviously, $\overline w_\Na \neq 0$ in this case). Then, applying Carath\'eodory's theorem again, we deduce that $p\leq \dim H_\mathcal{N}$ as a consequence of the minimality of $p$.

Define then
\begin{equation}\label{representation1}
\overline v_\mathcal{N} = \phi_\mathcal{N}(\overline u_\Na)\sum_{i=1}^p \alpha_i v_i = \sum_{i=1}^p \gamma_i v_i\, ,
\end{equation}
where $\gamma_i = \phi_\mathcal{N}(\overline u_\mathcal{N}) \alpha_i$.
From \eqref{beginn} and the linearity of $\mathcal{A}_\mathcal{N}$ we infer that 
\begin{equation}\label{oo}
\mathcal{A}_\mathcal{N} \overline v_\mathcal{N} = \mathcal{A}_\mathcal{N} \overline u_\mathcal{N} 
\end{equation}
and in particular
$\mathcal{A}_\mathcal{N} \overline v_\mathcal{N} \in \partial F_\mathcal{N}^* (-\overline w_{\mathcal{N}})$. In order to deduce that $\overline v_\mathcal{N}$ is a minimizer for Problem \ref{auxpro} using Proposition \ref{opt}, it remains to prove that
\begin{displaymath}
\phi_{\mathcal{N}}(\overline v_\mathcal{N}) = \langle \mathcal{A}_\mathcal{N}^* \overline w_{\mathcal{N}} , \overline v_\Na \rangle \, .
\end{displaymath}
Indeed, using $\phi_\Na(v_i) \leq 1$ for each $i=1,\ldots,p$, Equation
\eqref{oo} and Remark \ref{newopt}, we obtain
\begin{eqnarray*}
\phi_{\mathcal{N}}(\overline v_\mathcal{N}) &=& \phi_{\mathcal{N}}(\overline u_\mathcal{N}) \phi_{\mathcal{N}}\Big(\sum_{i=1}^p \alpha_i v_i\Big) \leq \phi_{\mathcal{N}}\left(\overline u_\mathcal{N}\right) \sum_{i=1}^p \alpha_i \phi_{\mathcal{N}}(v_i)\\
&\leq&\langle \mathcal{A}_\mathcal{N}^* \overline w_{\mathcal{N}} ,\overline u_\mathcal{N} \rangle = \langle \overline w_{\mathcal{N}} , \mathcal{A}_\mathcal{N} \overline u_\mathcal{N}\rangle \\
&=& \langle \overline w_{\mathcal{N}} , \mathcal{A}_\mathcal{N} \overline v_\mathcal{N}\rangle = \langle \mathcal{A}_\mathcal{N}^* \overline w_{\mathcal{N}} , \overline v_\mathcal{N} \rangle \, .
\end{eqnarray*}
On the other hand as $\mathcal{A}_\mathcal{N}^* \overline w_{\mathcal{N}} \in \mathcal{K}$ (see Remark \ref{newopt}) we have also that
\begin{displaymath}
\langle \overline v_\Na ,\mathcal{A}_\mathcal{N}^* \overline w_{\mathcal{N}} \rangle \leq \phi_{\mathcal{N}}(\overline v_\mathcal{N})\, .
\end{displaymath}
Thus, as a consequence of Proposition \ref{opt}, $\overline v_\mathcal{N}$ is a minimizer for Problem \ref{auxpro}. To conclude, notice that there exist $u_1,\ldots, u_p \in X$ such that $u_i + \mathcal{N} = v_i$
and therefore from \eqref{representation1} we have
\begin{displaymath}
\overline v_\mathcal{N} = \sum_{i=1}^p \gamma_i (u_i + \mathcal{N}) = \sum_{i=1}^p \gamma_i u_i + \mathcal{N}\,.
\end{displaymath}
Then applying Theorem \ref{ex} we infer that there exists $\overline u\in X$ such that $\overline u + \mathcal{N} = \overline v_\mathcal{N}$ and $\overline u$ is a minimizer of Problem \ref{problem}. From the equality
\begin{displaymath}
\overline u + \mathcal{N} = \sum_{i=1}^p \gamma_i u_i + \mathcal{N}
\end{displaymath} 
we obtain the existence of a $\overline \psi \in \Na$ such
that~\eqref{null} holds.
\end{proof}

\begin{rmk}
Let us point out similarities and differences to the work \cite{chambollerepresenter}, where a theorem similar to Theorem \ref{maint} has been shown. First, instead of seminorms, \cite{chambollerepresenter} deals with general convex regularizers. Moreover, in \cite{chambollerepresenter}, the existence of minimizers for the considered variational inverse problem is assumed \emph{a priori}, with the goal of disentangling the main result (which is purely geometric) from the topology chosen on $X$. In contrast, we make suitable assumptions that ensure existence of minimizers for the inverse problem and that the set of extremal points of the balls of the regularizer is non-empty. In such a way, we provide an operative result with hypotheses that can be easily checked.

\medskip

It is worth to notice that both our result and \cite{chambollerepresenter} do not provide a sparse representation for \emph{every} minimizer of the variational inverse problem. However, the points of view are complementary. In \cite{chambollerepresenter}, the authors characterize,
with a help of a theorem by Dubins and Klee \cite{dubins, klee},
the minimizers belonging to the finite-dimensional faces of the set of the solutions (we refer to \cite{chambollerepresenter} for the definition of the face of a convex set). In particular, when the dimension of a face is zero, i.e., the face is an extremal point, it is possible to obtain a sparse representation of the minimizer in terms of the extremal points and extremal rays of a certain sublevel set of the regularizer (see Section 2 in \cite{chambollerepresenter} for the definition of extremal ray). This is still true when the dimension of the face is larger than zero and finite (see Theorem 1 in \cite{chambollerepresenter}).
Existence of extremal points is then, e.g., obtained by Klee's extension of the Krein--Milman theorem \cite{klee2} in case of regularizers whose sublevel sets are closed, convex and locally compact in an appropriate locally convex space.
On the contrary, our theorem always provides the existence of a minimizer represented as a convex combination of extremal points of the ball of the regularizer. Due to the different techniques used, such a sparse minimizer does not necessarily belong to a finite-dimensional face of the set of the solutions.

\medskip Finally, let us point out that in order to obtain sparse
representations for solutions of the general variational
problem~\eqref{prointro}, the authors of \cite{chambollerepresenter}
consider solutions of the optimization problem
\begin{equation}\label{strong}
\min_{u\in X} \ \phi(u) \quad \text{subjected to  }\  \mathcal{A}u = y\,,
\end{equation} 
i.e., are forced
to pass from~\eqref{prointro} to~\eqref{strong}, then apply Klee--Dubins' theorem, and afterwards use sparse
solutions of \eqref{strong} to construct sparse solutions
of~\eqref{prointro}.
We remark that due to the use of different techniques, such a procedure is not required in our paper.
\end{rmk}

\section{Examples of sparsity for relevant regularizers}
In this section we study the structure of the extremal points for relevant regularizers, in order to applying the results of the previous section. The first example is about the Radon norm in the space of measures.
\subsection{The Radon norm for measures}\label{radonnormmeas}
Given $\Omega \subset \R^d$ a non-empty, open, bounded set, we set $X = \mathcal{M}({\Omega})$ the set of Radon measures on $\Omega$. 
We choose $\phi(u) = \|u\|_{\mathcal{M}}$ defined as
\begin{displaymath}
\|u\|_{\mathcal{M}} = \sup\left\{\int_{\Omega} \varphi\, du : \varphi \in C_c(\Omega),\ \|\varphi\|_\infty \leq 1\right\}\,. 
\end{displaymath}
Moreover we consider $F : H \rightarrow (-\infty, +\infty]$ satisfying the hypothesis given in Section \ref{assumptions}
and a linear continuous and surjective operator $\mathcal{A} : \mathcal{M}(\Omega) \rightarrow H$, where $H$ is a finite dimensional Hilbert space.

Under these choices we want to apply Theorem \ref{maint} to Problem \ref{problem}.
In this case
\begin{displaymath}
B = \{u \in \mathcal{M}(\Omega) : \|u\|_{\mathcal{M}}\leq 1\}
\end{displaymath}
and $\mathcal{N} = \{0\}$ such that $X_\Na = X$ and $B=B_\mathcal{N}$. It is standard to check that with these choices, all the hypotheses of Theorem \ref{maint} are verified. 

In order to get more information from Theorem \ref{maint} we need to characterize the extremal points of $B$.
This result is well-known, but we go through it for the reader's convenience.
\begin{prop}\label{dirac}
Given $B$ defined as above we have that
\begin{equation}
Ext(B) = \{\sigma \delta_x : x \in \Omega,\ \sigma \in \{-1,1\}\}\,.
\end{equation}
\end{prop}
\begin{proof}
Let us prove that $\delta_x, -\delta_x \in Ext(B)$ for every $x\in \Omega$. Indeed, let us suppose that there exists $u_1, u_2 \in B$ such that 
\begin{displaymath}
\delta_x = tu_1 + (1-t)u_2
\end{displaymath} 
for $t\in (0,1)$. Separating the positive part and negative part of $u_1$ and $u_2$, we can suppose without loss of generality that $u_1 \geq 0$ and $u_2 \geq 0$. Then $\mbox{supp}\, u_1 \subset \{x\}$ and hence $u_1 = \delta_x$. Similarly one can prove that $-\delta_x \in Ext(B)$ for every $x\in \Omega$.

On the other hand we prove that there are not other extremal points different from the Dirac deltas.
Suppose by contradiction that there exists an extremal point $u$ not supported on a singleton. Then $\|u\|_\mathcal{M} =1$ and there exists a measurable set $A \subset \Omega$ such that $0<|u|(A)<1$. We have
\begin{displaymath}
u = |u|(A) \left[\frac{1}{|u|(A)}u\res A\right]  + |u|(\Omega \setminus A)\left[\frac{1}{|u|(\Omega \setminus A)} u\res (\Omega \setminus A)   \right],
\end{displaymath}
which implies that $u$ is not an extremal point.
Hence all the extremal points of $B$ are of the form $a\delta_x$ where $a \in \R$ and $x\in \Omega$. As the extremal points of $B$ have unit Radon norm we deduce immediately that $|a| = 1$.
\end{proof}

From Proposition \ref{dirac} we obtain immediately the following theorem:

\begin{thm}
Under the previous choices of $X$, $\phi$, $\mathcal{A}$ and $F$, there exists a minimizer of Problem \ref{problem} denoted by $\overline u \in X$ such that
\begin{displaymath}
\overline{u} = \sum_{i=1}^p \gamma_i \delta_{x_i}\, ,
\end{displaymath}
where $p\leq \mbox{dim}\,H$, $\gamma_1,\ldots,\gamma_p \in \R \setminus \{0\}$, $x_1,\ldots,x_p \in \Omega$ and $\sum_{i=1}^p |\gamma_i| = \|\overline{u}\|_\M$.
\end{thm}

\subsection{The total variation for BV functions}
Let $\Omega \subset \R^d$ be a non-empty, bounded Lipschitz domain. We want to apply the result of the previous section for $X = BV(\Omega)$ and $\phi(u) = |Du|(\Omega)$, where 
\begin{equation}
|Du|(\Omega) := \sup\left\{\int_{\Omega} u \, \div \varphi\, dx : \varphi \in C_c^1(\Omega), \ \|\varphi\|_\infty \leq 1\right\}\,. 
\end{equation}

This is a relevant setting for inverse problems in image processing as $\phi(u) = |Du|(\Omega)$ is the classical TV regularizer which has widely been studied and used in applications. We refer to \cite{FOB} for the basic definitions regarding BV functions and sets of finite perimeter that we will 
use. 

We equip $X$ with the weak* topology for BV functions by interpreting $BV(\Omega)$ as a dual space (see, for instance \cite[Remark 3.12]{FOB}).
As in the previous example we consider a linear, continuous and surjective map $\mathcal{A} : BV(\Omega) \rightarrow H$ and $F: H \rightarrow (-\infty,+\infty]$ that satisfies the assumptions given in Section \ref{assumptions}. Under these choices we want use Theorem \ref{maint} to characterize the sparse solutions of Problem \ref{problem}.

Notice that with the chosen topology on $X$, the functional $\phi(u) = |Du|(\Omega)$ is a lower semi-continuous seminorm and $\mathcal{A}$ satisfies assumption \ref{H0}. Therefore in order to apply Theorem \ref{maint} we just need to verify Hypothesis \ref{H1} that is the content of the next lemma. Notice that in this specific case, we have $\mathcal{N} = \R$ as $\Omega$ is connected.

\begin{lemma}\label{netlemma}
Defining $\phi_\mathcal{N}(u_\mathcal{N}) := |Du|(\Omega)$, the sublevel sets
\begin{displaymath}
S^-(\phi_{\mathcal{N}},\alpha):=\{u_\mathcal{N} \in X_{\mathcal{N}} : \phi_{\mathcal{N}}(u_\mathcal{N})\leq \alpha\}
\end{displaymath}
are compact for every $\alpha > 0$.
\end{lemma}
\begin{proof}
We first remark that the metrizability of the space $X_\mathcal{N}$ on bounded sets is not straightforward to show. Therefore we work with \emph{nets} instead of sequences (we refer to Sections 1.3, 1.4, 1.6 in  \cite{pedersen} for the basic properties of nets).

Consider a net $(u_\mathcal{N}^\beta)_\beta \subset X_\mathcal{N}$ such that $|Du^\beta|(\Omega) \leq \alpha$. Using the Poincar\'e inequality for BV functions (see Theorem 3.44 in \cite{FOB}) we deduce that there exists $c^\beta \in \R$ such that 
\begin{equation}\label{rightconstant}
\|u^\beta + c^\beta\|_{BV} = \|u^\beta + c^\beta\|_{L^1} + |Du^\beta|(\Omega) \leq C(\Omega) |Du^\beta|(\Omega) + |Du^\beta|(\Omega) \leq \alpha(C(\Omega) + 1)
\end{equation}
for every $\beta$. Recall now that bounded sets of BV functions are compact with respect to weak* convergence of nets (as mentioned earlier, %
the space of BV functions is isomorphic to the dual of a separable Banach space according to Remark 3.12 in \cite{FOB}; this implies compactness of weak*-closed bounded sets by the Banach--Alaoglu theorem). So, thanks to \eqref{rightconstant}, there exists a subnet (not relabelled) $(u^{\beta} + c^{\beta})_\beta$ in $BV(\Omega)$ and $u \in BV(\Omega)$ such that $u^{\beta} + c^{\beta} \rightarrow u$ in $L^1(\Omega)$ and $D u^{\beta} \stackrel{*}{\rightharpoonup} Du$ in $\M(\Omega,\R^d)$ in the sense of nets (see Theorem 1.6.2 in \cite{pedersen}).  By the continuity of the projection on the quotient we obtain that 
\begin{equation*}
u^{\beta} + \mathcal{N} \rightarrow u + \mathcal{N}  \qquad \mbox{ in } \ X_{\mathcal{N}},
\end{equation*}
see Proposition 1.4.3 in \cite{pedersen}.
Then, by the lower semi-continuity of $|Du|(\Omega)$ we have that the sublevel sets of $|Du|(\Omega)$ are weak*-closed. %
This implies that $u\in S^{-}(\phi_\mathcal{N}, \alpha)$ (Proposition 1.3.6 in \cite{pedersen}).
\end{proof}

In order to have an explicit representation for the sparse minimizer we aim to characterize the extremal points of the set 
\begin{equation}
B_\mathcal{N} = \{u \in BV(\Omega): |D u|(\Omega)\leq 1\} + \mathcal{N} \subset X_\Na \,.
\end{equation}
The result is known for $\Omega = \R^d$ and it was proved in \cite{Flemingext} for $d=2$ and then extended to all dimensions in \cite{fleming2} and in a slighly different setting in \cite{Ambrosiocasellesconnected}. 
Our plan is to modify the approach in \cite{Ambrosiocasellesconnected} for the case of $\Omega$ bounded.
The first definition is taken from \cite{Ambrosiocasellesconnected} and it is a suitable modification of the classical definition for currents given in \cite{GMT}. Recall that a set of \emph{finite perimeter} is a measurable set $A \subset \Omega$ such that $\chi_A \in BV(\Omega)$ for the characteristic function $\chi_A$ of $A$. In this case, we call $P(A,\Omega) = |D\chi_A|(\Omega)$ the \emph{perimeter} of $A$.

\begin{defi}[Decomposable set]\label{dec}
A set of finite perimeter $E \subset \Omega$ is \emph{decomposable} if there exists a partition of $E$ in two sets $A$, $B$ with $|A| > 0$ and $|B| > 0$ such that $P(E,\Omega) = P(A,\Omega) + P(B,\Omega)$. A set of finite perimeter is \emph{indecomposable} if it is not \emph{decomposable}.
\end{defi}

In \cite{Ambrosiocasellesconnected}, the notion of \emph{saturated} set is introduced that is suitable in the case $\Omega = \R^d$. In our case of bounded domains, we do not need this requirement, but we ask that both the set and its complement are indecomposable.

\begin{defi}[Simple set]\label{simpleset}
We say that a set of finite perimeter $E$ is \emph{simple} if both $E$ and $\Omega \setminus E$ are \emph{indecomposable}.
\end{defi}

In what follows we denote by $E^1$ the measure theoretic interior of $E$ defined as
\begin{equation*}
E^1 := \left\{x \in \R^d : \lim_{r \rightarrow 0} \frac{|E \cap B_r(x)|}{|B_r(x)|} = 1\right\}
\end{equation*}
and by $E^0$ the measure theoretic exterior:
\begin{equation*}
E^0 := \left\{x \in \R^d : \lim_{r \rightarrow 0} \frac{|E \cap B_r(x)|}{|B_r(x)|} = 0\right\}\, .
\end{equation*}
The essential boundary of $E$ is then defined as $\partial^* E = \R^d \setminus (E^0 \cup E^1)$.

We will also need the following result due to Dolzmann and M\"uller \cite{dolzmann}. 

\begin{lemma}[Constancy theorem]\label{Dolz}
Given $u\in BV(\Omega)$ and $E \subset \Omega$ an indecomposable set such that
\begin{displaymath}
|Du|(E^1) = 0\, ,
\end{displaymath}
then there exists $c \in \R$ such that $u(x) = c$ almost everywhere in $E$.
\end{lemma}

With the following theorem we are able to characterize the extremal points of $B_\mathcal{N}$ in a rather straightforward way without relying on indecomposability results for the reduced boundary as in \cite{fleming2}.
\begin{thm}
We have that
\begin{displaymath}
Ext (B_\mathcal{N}) = \left\{\frac{\chi_E}{P(E, \Omega)} + \mathcal{N} : E \mbox{ simple}\right\}\,.
\end{displaymath}
\end{thm}
\begin{proof}
We start to prove that
\begin{displaymath}
Ext (B_\mathcal{N}) \subset \left\{\frac{\chi_E}{P(E, \Omega)} + \mathcal{N} : E \mbox{ simple}\right\}\,.
\end{displaymath}
Taking $u_\mathcal{N} \in Ext (B_\mathcal{N})$ and choosing $u \in BV(\Omega)$ such that $u + \Na = u_\Na$, we have clearly that $|Du|(\Omega) = 1$. We want to show that $u$ assumes two values almost everywhere. In order to do that 
 we define
\begin{displaymath}
F(s) = \int_{-\infty}^s  P(\{u(x) \leq t\}, \Omega) \, dt\,.
\end{displaymath}
We have $F(-\infty) = 0$ and by the coarea formula for BV functions,
$F(+\infty) = 1$. Moreover, the function
$t \mapsto P(\{u(x) \leq t\}, \Omega)$ is integrable on $\R$, so
there exists an $s \in \R$ such that $F(s) = \tfrac12$. Setting
\[
u_1 = 2 \min(u, s), \qquad u_2 = 2 \max(u - s, 0)
\]
we see that $u = \tfrac12 u_1 + \tfrac12 u_2$ as well as
$|Du_1|(\Omega) = |Du_2|(\Omega) = 1$, the latter again by the coarea
formula and the choice of $s$. As $u_\Na$ is an extremal point of
$B_\Na$, it follows that $u_\Na = (u_1)_\Na = (u_2)_\Na$ which means
that there exist $c_1,c_2 \in \R$ such that
$u = u_1 + c_1 = u_2 + c_2$. Now, for $x \in \Omega$ such that
$u(x) \geq s$, this implies $u(x) = 2 s + c_1$. Likewise, if
$u(x) \leq s$, then $u(x) = c_2$. Hence, $u$ assumes at most two
values almost everywhere. However, since $|Du|(\Omega) = 1$, it
assumes exactly two values almost everywhere and $2s + c_1 >
c_2$. (Moreover, the set $\{u(x) = s\}$ must be a null set.)

Up to change of the representative $u$ of $u_\mathcal{N}$, we can suppose that $u(x) \in \{0,a\}$ almost everywhere, where $a > 0$. Defining $E = \{x\in \Omega : u(x) = a\}$ and using the fact that $|Du|(\Omega) = 1$ one concludes that
$u= \frac{\chi_E}{P(E,\Omega)}$.
Suppose now by contradiction that $E$ is decomposable and let $A$ and
$B$ be the sets of finite perimeter given by Definition
\ref{dec}. Then, $P(A,\Omega) > 0$ and $P(B,\Omega) > 0$ and defining
\begin{equation*}
u_1 = \frac{\chi_A}{P(A,\Omega)} \quad \mbox{and} \quad u_2 = \frac{\chi_B}{P(B,\Omega)}\,,
\end{equation*}
we have
\begin{equation}\label{abl}
u= \frac{\chi_E}{P(E,\Omega)} = \frac{P(A,\Omega)}{P(E,\Omega)}u_1 + \frac{P(B,\Omega)}{P(E,\Omega)}u_2\, .
\end{equation}
Hence by the properties of $A$ and $B$ given by Definition \ref{dec}, Formula~\eqref{abl} is a non-trivial convex combination of $u$.

Likewise, suppose by contradiction that $\Omega \setminus E$ is decomposable and call $A$ and $B$ its decomposition according to Definition \ref{dec}. Define
\begin{equation*}
u_1 = -\frac{\chi_{A}}{P(A,\Omega)} \quad \mbox{and} \quad u_2 = \frac{1 - \chi_{B}}{P(B,\Omega)}\,. 
\end{equation*}
Notice that
\begin{equation}\label{decc}
u= \frac{\chi_E}{P(E,\Omega)} = \frac{P(A,\Omega)}{P(E,\Omega)}u_1 + \frac{P(B,\Omega)}{P(E,\Omega)}u_2\,.
\end{equation}
So using that $\{A,B\}$ is a decomposition of $\Omega \setminus E$ and the fact that $P(E,\Omega) = P(\Omega \setminus E, \Omega)$ we conclude that \eqref{decc} is a non-trivial convex combination of $u$.

Thus, $E$ must be a simple set and the first inclusion is proven.

\medskip

Let us prove now the opposite inclusion: 
\begin{displaymath}
Ext (B_\mathcal{N}) \supset \left\{\frac{\chi_E}{P(E\cap \Omega)} + \mathcal{N} : E \mbox{ simple}\right\}\,.
\end{displaymath}
Given $E \subset \Omega$ a simple set, %
let us suppose that there exists $u_1,u_2 \in BV(\Omega)$ such that $|D u_1|(\Omega) \leq 1$, $|D u_2|(\Omega) \leq 1$ and
\begin{displaymath}
\frac{\chi_E}{P(E, \Omega)} + \mathcal{N} = \lambda (u_1 + \Na) + (1-\lambda)(u_2 + \Na)\,, 
\end{displaymath}
where $\lambda \in (0,1)$. This means that there exists $c \in \R$ such that
\begin{displaymath}
\frac{\chi_E}{P(E, \Omega)} + c = \lambda u_1 + (1-\lambda)u_2\, 
\end{displaymath}
and so
\begin{displaymath}
\frac{D\chi_E}{P(E, \Omega)} =\lambda D u_1 + (1-\lambda)Du_2\,.
\end{displaymath}
Notice that for every $A \subset \Omega$ measurable one has
\begin{equation}\label{split}
\frac{|D\chi_E|(A)}{P(E,\Omega)} = \lambda |D u_1|(A) + (1-\lambda)|Du_2|(A)\, .
\end{equation}
Indeed, if there exists $A \subset \Omega$ such that $\lambda |D u_1|(A) + (1-\lambda)|Du_2|(A) > \frac{|D\chi_E|(A)}{P(E,\Omega)}$ we would arrive at the contradiction
\begin{eqnarray*}
1 & = & \frac{|D\chi_E|(\Omega)}{P(E,\Omega)} = \frac{|D\chi_E|(A)}{P(E,\Omega)}  + \frac{|D\chi_E|(A^c)}{P(E,\Omega)} < \lambda |D u_1|(A) + (1-\lambda)|Du_2|(A) + \frac{|D\chi_E|(A^c)}{P(E,\Omega)}\\
& \leq & \lambda |D u_1|(A) + (1-\lambda)|Du_2|(A) +  \lambda |D u_1|(A^c) + (1-\lambda)|Du_2|(A^c)  \leq 1\,.
\end{eqnarray*}
As derivative of the characteristic function of a set of finite
perimeter, $D\chi_E$ can only be supported on the reduced boundary
$\partial^*E$. Thus, $|D\chi_E|(E_0) = |D\chi_E|(E_1) = 0$
and~\eqref{split} gives
$|Du_1|(E_0) = |Du_2|(E_0) = |Du_1|(E_1) = |Du_2|(E_1) = 0$. Applying
Lemma~\ref{Dolz} with the indecomposable sets $E$ and $\Omega \setminus E$ then
yields that $u_i = d_i \chi_E + c_i$ for some $c_i,d_i \in \R$,
$i=1,2$. By~\eqref{split}, we further deduce
$|Du_1|(\Omega) = |Du_2|(\Omega) = 1$ which implies that
$|d_1| = |d_2| = P(\Omega,E)^{-1} > 0$. Clearly, $d_1$ and $d_2$
cannot both be negative. Also, $d_1$ and $d_2$ cannot have opposite
sign as in this case, %
comparing $|D\chi_E|(\Omega)/P(E,\Omega)$ and
$|\lambda Du_1 + (1-\lambda)Du_2|(\Omega)$ leads to the contradiction
\[
  1 = |\lambda d_1 + (1-\lambda d_2)|P(E,\Omega) < (\lambda |d_1| +
  (1-\lambda) |d_2|) P(E,\Omega) = 1\,.
\] 
Hence, $d_1 = d_2 = P(E,\Omega)^{-1}$ and 
\[
  \frac{\chi_E}{P(E, \Omega)} + c = \frac{\chi_E}{P(E, \Omega)} + c_1
 = \frac{\chi_E}{P(E, \Omega)} + c_2\,.
\]
In other words,
$\frac{\chi_E}{P(E, \Omega)} + \Na = u_1 + \Na = u_2 + \Na$, so
$\frac{\chi_E}{P(E, \Omega)} + \Na$ is indeed an extremal point.
\end{proof}

We have shown the following theorem.
\begin{thm}
\label{sparsity_tv}
If $X = BV(\Omega)$ and $\phi(u) = |Du|(\Omega)$ there exists a minimizer $\overline u \in BV(\Omega)$ of Problem~\eqref{problem}  such that
\begin{equation}
\overline{u}  = c + \sum_{i=1}^p \frac{\gamma_i}{P(E_i,\Omega)} \chi_{E_i} \, ,
\end{equation}
where $c \in \R$, $p\leq \mbox{dim}\,(H / \mathcal{A}(\R))$,  $\gamma_i >0$ with $\sum_{i=1}^p \gamma_i = |D\overline u|(\Omega)$ and each $E_i \subset \Omega$ is simple.
\end{thm}

\subsection{Radon norm of a scalar differential operator}
In this section we consider the case where $\phi(u) = \|Lu\|_{\mathcal{M}}$, namely the Radon norm of a linear, translation-invariant scalar differential operator $L$. This was already treated in \cite{unsersplines} and in \cite{exactsolutionsflinth} in different settings. Our goal is to show that our theory applies straightforwardly to this case. We start some useful properties of scalar differential operators that we are going to use.
In what follows we denote by $\alpha = (\alpha_1,\ldots,\alpha_d) \in \N^{d}$ a multi-index and we employ the standard multi-index notation and conventions.

\subsubsection{Some technical lemmas}

We consider a non-zero differential operator with linear coefficients of order $q \in \N$ of the form

\begin{equation}\label{diffop}
L = \sum_{|\alpha| \leq q} c_\alpha \partial^\alpha\,,
\end{equation}
where each $c_\alpha \in \R$ and $c_\alpha \neq 0$ for some $|\alpha|=q$. We also denote by $L^*$ the operator defined formally by 
\begin{equation}\label{adjoint}
L^* = \sum_{|\alpha| \leq q} (-1)^ {|\alpha|} c_\alpha \partial^\alpha\,.
\end{equation}

The existence of a fundamental solution $G$ for $L$ is ensured by
virtue of the classical Malgrange--Ehrenpreis theorem (see for example
Theorem 8.5 in \cite{Rudin}).

\begin{thm}[Malgrange--Ehrenpreis]
Given $L$ a non-zero differential operator with linear coefficients according to~\eqref{diffop} there exists a distribution $G \in D(\R^d)^*$ which is a fundamental solution for $L$, namely
\begin{equation}\label{funddirac}
LG = \delta_0 \quad \mbox{ in } D(\R^d)^*\,.
\end{equation}
\end{thm}
Let $\Omega \subset \R^d$ be a non-empty open and bounded set.
Let us define the operator $T : \mathcal{M}(\Omega) \rightarrow D(\R^d)^*$ as $T\mu = \widetilde \mu \star G$, where
\begin{displaymath}
\widetilde \mu (A) = \mu (\Omega \cap A)
\end{displaymath} 
for every Borel set $A\subset \R^d$ and $\star$ denotes the convolution of a compactly supported distribution and a distribution. 
Notice that $T$ is indeed well-defined because $\widetilde \mu$ is compactly supported on $\R^d$ and $\widetilde \mu \star G \in D(\R^d)^*$. 
Define then $T_\Omega : \mathcal{M}(\Omega) \rightarrow D(\Omega)^*$ as $T_\Omega(\mu) = (T\mu)_{|_{\Omega}}$.

\begin{rmk}\label{involution}
Notice that $LT_\Omega\mu = \mu$ for every $\mu \in \mathcal{M}(\Omega)$. Indeed,
\begin{equation}
LT_\Omega\mu = L(T \mu)_{\vert_{\Omega}} = (LT\mu)_{\vert_{\Omega}} = (L(\widetilde \mu\star G))_{\vert_{\Omega}} = (\widetilde \mu\star LG)_{\vert_{\Omega}} = \mu\,,
\end{equation}
where in the last equality we use \eqref{funddirac}.
\end{rmk}

\begin{lemma}\label{neworder}
There exists $C \in \R$ and $s \in \N$ such that for every $\mu\in \mathcal{M}(\Omega)$ one has
\begin{displaymath}
|(T_\Omega \mu)(\varphi)| \leq C\|\mu\|_{\mathcal{M}}\sup\{|\partial^\alpha \varphi(x)| : x \in \Omega, \ |\alpha|\leq s\}
\end{displaymath}
for every test function $\varphi\in D(\Omega)$. In particular, for
each $\mu \in \M(\Omega)$, $T_\Omega\mu$ can be extended to a unique element
in $C^s_0(\Omega)^*$ such that $T_\Omega: \M(\Omega) \to C_0^s(\Omega)^*$ becomes
a linear and continuous mapping.
\end{lemma}
\begin{proof}
Consider a test function $\varphi \in D(\Omega)$ and denote by $\widetilde \varphi$ its zero extension to $\R^d$. Then, the order of $G$ is finite on bounded sets which means that there exists $s\in \N$ such that
\begin{eqnarray*}
|(T_\Omega \mu)(\varphi)|= |T \mu (\widetilde{\varphi})| &=& |(\widetilde \mu \star G) (\varphi)|  = \left|G \left(x\mapsto\int_{\R^d} \widetilde{\varphi}(x+y) \, d\widetilde \mu(y)\right)\right| \\
&\leq & C \sup\left\{|\partial^\alpha \psi(x)| : x \in \R^d, \ |\alpha|\leq s\right\} \, ,
\end{eqnarray*}
where we set
\begin{displaymath}
\psi(x) = \int_{\R^d} \widetilde{\varphi}(x+y) \, d\widetilde \mu(y)\,,
\end{displaymath}
whose support is contained in a compact set that only depends on $\Omega$.
Notice now that for every $x \in \R^d$ we have
\begin{equation*}
|\partial^\alpha \psi(x)| = \left| \int_{\R^d} \partial^\alpha \widetilde{\varphi}(x+y) \, d\widetilde \mu(y)\right|\leq \|\mu\|_{\mathcal{M}} \sup_{x \in \Omega} |\partial^\alpha \varphi(x)|\,.
\end{equation*}
So 
\begin{equation*}
|(T_\Omega \mu)(\varphi)| \leq C \|\mu\|_{\mathcal{M}} \sup\{|\partial^\alpha \varphi(x)| : x \in \Omega, \ |\alpha|\leq s\}\,,
\end{equation*}
meaning that $T_\Omega \mu$ can be extended, by density, to an element
in $C_0^s(\Omega)^*$. The latter also establishes the claimed
continuity of $T_\Omega: \M(\Omega) \to C_0^s(\Omega)^*$.
\end{proof}

\subsubsection{Existence of a sparse minimizer}

Recall that we consider the differential operator $L$ given in Equation \eqref{diffop}. With $s$ is given by Lemma \ref{neworder}, we set $X = C_0^s(\Omega)^*$, the space of distributions of order $s$ equipped with the weak* topology. From now on we consider the weak differential operator $L$ mapping between $X \rightarrow C_0^{s+q}(\Omega)^*$. Notice that with this definition, $L$ is a continuous operator when $X$ and $C_0^{s+q}(\Omega)^*$ are equipped with the weak* topology. Indeed, the adjoint $L^*$ according to~\eqref{adjoint} maps continuously between the spaces $C_0^{s+q}(\Omega) \to C_0^s(\Omega)$ as a classical differential operator. Thus, considering $u_n \stackrel{*}{\rightharpoonup} u$ in $X$ and $\varphi \in C_0^{s+q}(\Omega)$ we have $L^*\varphi \in C_0^{s}(\Omega)$ and hence,
\begin{equation}
\lim_{n\rightarrow + \infty} Lu_n(\varphi) = \lim_{n\rightarrow + \infty} u_n(L^*\varphi) = Lu(\varphi)\,,
\end{equation}
which establishes the weak*-continuity as due to separability of $C_0^s(\Omega)$ and $C_0^{s+q}(\Omega)$, it suffices to consider sequences.

\medskip
We then define the following functional $\phi : X \rightarrow [0, +\infty]$: 
\begin{equation}\label{defseminorm}
\phi(u) := \left\{ \begin{array}{ll}
\|Lu\|_{\mathcal{M}} & \mbox{ if } \|Lu\|_{\mathcal{M}}< +\infty\\
+\infty & \mbox{ otherwise }\,.
\end{array} 
\right.
\end{equation}

\begin{rmk}\label{lscweakstar}
Notice that $\phi$ is a seminorm and it is lower semi-continuous in $X$ (with respect to the weak* topology). Indeed, once again, as $C_0^s(\Omega)$ is separable we know that weak* lower semi-continuity for $L$ is equivalent to weak* sequential lower semi-continuity. Therefore, we consider a sequence $(u_n)_n \subset X$ such that $u_n \stackrel{*}{\rightharpoonup} u$ in $X$ and we suppose without loss of generality that 
\begin{equation*}
\liminf_{n\rightarrow + \infty} \|Lu_n\|_{\mathcal{M}} < +\infty \quad \mbox{and} \quad \lim_{n\rightarrow +\infty} \|Lu_n\|_{\mathcal{M}}  = C\,.
\end{equation*}
Then, by weak* sequential compactness of measures there exists $v \in \mathcal{M}(\Omega)$ such that, up to subsequences, $Lu_n \stackrel{*}{\rightharpoonup} v$ in $\mathcal{M}(\Omega)$ and in particular in $C_0^{s+q}(\Omega)^*$. As $L$ is weak*-weak* closed we infer that $v = Lu$ and from the lower semi-continuity of the Radon norm with respect to weak* convergence in $\mathcal{M}(\Omega)$ we conclude that $\|Lu\|_{\mathcal{M}} \leq C$.

\end{rmk} 

In order to apply Theorem \ref{maint},
it remains to verify Assumption \ref{H1}. This is the content of the next proposition. We remind that $X_\mathcal{N} = X + \mathcal{N}$ (equipped with the quotient of the weak* topology of $X$) where $\mathcal{N}$ is the null-space of $L$ and $\phi_\mathcal{N}(u_\mathcal{N}) = \phi_\mathcal{N}(u + \mathcal{N}) := \phi(u)$ (for notational convenience we denote by $L$ the operator acting on $X_\mathcal{N}$ in the natural way).

\begin{prop}\label{par}
The sublevel sets of $\phi_{\mathcal{N}}$, i.e.,
\begin{equation*}
S^-(\phi_{\mathcal{N}},\alpha):=\{u_\mathcal{N} \in X_{\mathcal{N}} : \phi_{\mathcal{N}}(u_\mathcal{N})\leq \alpha\}\,,
\end{equation*}
are compact for every $\alpha > 0$.
\end{prop}
\begin{proof}
Similarly to the proof of Lemma \ref{netlemma} we employ \emph{nets} since metrizability of the space $X_\mathcal{N}$ does not play a role in this context.

Given $u_\mathcal{N} \in S^-(\phi_\Na,\alpha)$ we have thanks to Remark~\ref{involution} that
\begin{displaymath}
LT_\Omega Lu_{\mathcal{N}} = Lu_\mathcal{N}\,.
\end{displaymath}
Therefore, there exists $\psi \in \mathcal{N}$ such that $u + \psi = T_\Omega Lu_\Na$, where $u \in X$ is such that $u + \Na = u_\Na$.
Moreover, with the help of Lemma~\ref{neworder}, it follows that
\begin{equation*}
|(u + \psi)(\varphi)| = |(T_\Omega Lu_\Na)(\varphi)| \leq C\|Lu_\Na\|_{\mathcal{M}} \sup\,\{|\partial^\alpha \varphi(x)| : x \in \Omega, \ |\alpha|\leq s\}
\end{equation*}
for every $\varphi \in C_0^s(\Omega)$. Hence,
\begin{equation}\label{boundquot}
\inf_{\psi \in \mathcal{N}}\|u + \psi\|_{X} \leq C\|Lu_\mathcal{N}\|_{\mathcal{M}}\,.
\end{equation}
Consider now a net $(u_\mathcal{N}^\beta)_\beta \subset S^-(\phi_{\mathcal{N}},\alpha)$. Since $\phi_\Na(u_\Na) = \|Lu_\Na\|_\M$, we have
\begin{equation*}
\inf_{\psi \in \mathcal{N}}\|u^\beta + \psi\|_{X} \leq C\alpha\,.
\end{equation*}
for every $\beta$. Thus, there exists a net
$(\psi^\beta)_\beta$ in $\mathcal{N}$ and $\tilde C > 0$ such that
\begin{equation*}
\|u^\beta + \psi^\beta\|_{X}  \leq \tilde C\,.
\end{equation*}
Applying the Banach--Alaoglu theorem we extract a subnet (not relabelled) of $(u^\beta + \psi^\beta)_\beta$ that is converging to $u \in X$ in the weak* topology of $X$ (Theorem 1.6.2 in \cite{pedersen}). As the projection on the quotient is a continuous operation we deduce also that 
\begin{equation*}
u^\beta + \psi^\beta + \mathcal{N} \rightarrow u + \mathcal{N} = u_\Na \qquad \mbox{ in }\ X_\mathcal{N}\,,
\end{equation*}
(Proposition 1.4.3 in \cite{pedersen}).
It remains to show that $u_\mathcal{N} \in S^{-}(\phi_\mathcal{N} , \alpha)$.
Thanks to Remarks \ref{lscweakstar} and~\ref{coerclsc}, the functional $\phi_\Na  : X_\Na \rightarrow [0,\infty]$ is lower semi-continuous with respect to the quotient topology in $X_\Na$ and therefore, its sublevel sets are closed. This implies that $u_\Na \in S^{-}(\phi_{\mathcal{N}},\alpha)$ (Proposition 1.3.6 in \cite{pedersen}). 
\end{proof}

We are now in position to apply Theorem \ref{maint}. Consider $\mathcal{A} : X \rightarrow H$ a linear continuous operator such that \ref{H0} holds and $F:H \rightarrow (-\infty,+\infty]$ satisfying the assumptions in Section \ref{assumptions}.  

We set the following variational problem:
\begin{equation}\label{specific}
\inf_{u \in C_0^s(\Omega)^*} \|Lu\|_\M + F(\mathcal{A}u)\,.
\end{equation}
Thanks to Proposition \ref{par}, Theorem \ref{maint} is applicable. We can furthermore characterize the extremal points of the ball associated to $\phi_\mathcal{N}$ according to the following theorem. Note that a similar result was also obtained by \cite{unsersplines} and  \cite{exactsolutionsflinth} in different settings and more restrictive hypotheses. For this purpose, for $x \in \R^d$, denote by $G_x$ the fundamental solution $G$ translated by $x$, i.e., such that $LG_x = \delta_x$. 

\begin{thm}
\label{sparsity_diffop}
There exists $\overline u \in C_0^s(\Omega)^*$ a minimizer of \eqref{specific} with the following representation:
\begin{equation}\label{null2}
\overline{u}  = \overline{\psi} + \sum_{i=1}^p \gamma_i G_{x_i} \, ,
\end{equation}
where $\overline{\psi} \in C_0^s(\Omega)^*$ with $L\overline{\psi} = 0$, 
 $p\leq \mbox{dim}\,H_\mathcal{N}$, $x_1,\ldots,x_p \in \Omega$, and $\gamma_1,\ldots,\gamma_p \in \R \setminus \{0\}$ with $\sum_i |\gamma_i| = \|L\overline u\|_{\mathcal{M}}$.
\end{thm}

\begin{proof}
With $\pi_\Na$ denoting the quotient map $X \to X_\Na$, we have due to Remark~\ref{involution} that
\begin{eqnarray*}
B_\mathcal{N} & =  & \{u_\mathcal{N} \in X_\mathcal{N}: \|Lu_\Na\|_\mathcal{M} \leq 1\} \\
&=& \pi_{\mathcal{N}}(\{u\in X: \|Lu\|_{\mathcal{M}}\leq 1\}) \\
&=& (\pi_{\mathcal{N}} \circ T_{\Omega})(\{\mu \in \M(\Omega) : \|\mu\|_{\mathcal{M}} \leq 1\})\,.
\end{eqnarray*}
Notice that $\pi_\mathcal{N} \circ T_{\Omega} : \mathcal{M}(\Omega) 
\rightarrow X_\mathcal{N}$ is a linear, injective map.
Indeed, let us suppose that $(\pi_\mathcal{N} \circ T_{\Omega})(\mu)  = 0$. Then there exists $\psi \in \mathcal{N}$ such that $T_\Omega\mu = \psi$. Applying $L$ on both sides and using Remark \ref{involution} we deduce that $\mu = 0$.

Hence we can apply part \ref{fund_ii} of Lemma \ref{fund} to obtain
\begin{equation*}
Ext(B_\mathcal{N}) = (\pi_\mathcal{N} \circ T_\Omega) Ext( \{\mu \in \mathcal{M}(\Omega): \|\mu\|_{\mathcal{M}} \leq 1\})
\end{equation*}
and by Proposition \ref{dirac},
\begin{equation*}
Ext(B_\mathcal{N}) = (\pi_\mathcal{N} \circ T_\Omega) \{\sigma \delta_x : x \in \Omega,\ \sigma \in \{-1,1\}\}\,.
\end{equation*}
So applying Theorem \ref{maint} and noting that $T_\Omega \delta_x = G_x$ one concludes.
\end{proof}

\section{Conclusions and open problems}

The abstract main result of this paper contained in Theorem \ref{maint} about the structure of a minimizer of a variational problem with finite dimensional data appears to be widely applicable, thanks to its generality. The usability of this theorem %
to concrete problems relies, however, on the characterization of the extremal points of the unit ball associated with the given regularizer. Such a characterization appears to be fundamental for devising suitable algorithms that rely on the structure of the minimizers given by Theorem \ref{maint}.

In this paper we essentially carried out this characterization
for two specific regularizers:
\begin{itemize}
\item The total variation of a function with bounded variation.
\item The Radon norm of a scalar differential operator.
\end{itemize} 
In the meantime, a follow-up paper also provides the characterization of extremal points for the Benamou--Brenier energy in optimal transport \cite{extremalpointsbenamou}.
A challenging direction of further research is the study of the extremal points of balls associated with other classes of regularizers. For example, it would be of great interest to be able to treat the case of the Radon norm for general vector-valued differential operators. This would lead to the consideration, as an instance among others, of $TV^2$ regularization (see for example \cite{chambollelions})  which is defined as
\begin{equation*}
TV^2(u) = \sup\left\{\int_{\Omega} \nabla u \cdot \div \varphi \, dx : \varphi \in C^1_c(\Omega, \R^{d \times d}),\ \|\varphi\|_\infty \leq 1\right\}\,, 
\end{equation*}
that is the total variation of the weak gradient of an $L^1$ function. As a consequence, it would be possible to compare the regularizing effect of the $TV^2$ seminorm and the $TV$ seminorm, leading, e.g., to a better understanding of how higher-order regularizers reduce the staircase effect.

Additionally, one can also consider more complex regularizers that were studied to overcome the limitations of $TV$ and $TV^2$ models. For example, in \cite{totalgeneralizedvariation}, the so called \emph{total generalized variation} was introduced, which is defined in the following way:
\begin{equation*}
TGV_\alpha^k(u) = \sup\left\{\int_{\Omega} u \,\div^k \varphi \, dx : \varphi \in C_c^k(\Omega, \mbox{Sym}^k(\R^d)),\ \|\div^\ell \varphi\|_\infty \leq \alpha_\ell,\ \ell=0,\ldots,k-1\right\}\, ,
\end{equation*}
where $\mbox{Sym}^k(\R^d)$ is the space of symmetric tensors of order $k$ and $\alpha = (\alpha_0,\ldots, \alpha_{k-1})$ are positive parameters. The characterization of extremal points of the ball associated with these particular regularizers is, up to our knowledge, still not known and would lead to a deep understanding of the regularization effects in respective variational models. 

\subsection*{Acknowledgements}

The authors gratefully acknowledge the funding of this work by the
Austrian Science Fund (FWF) within the project P 29192. We also thank Professor Luigi Ambrosio for the useful remarks regarding \cite{Ambrosiocasellesconnected}.

\bibliographystyle{abbrv}
\bibliography{biblio}

\end{document}